\documentclass[10pt]{amsart}

\usepackage{amssymb,amsfonts,lineno,amsthm,amscd,stmaryrd}
\usepackage[leqno]{amsmath}
\usepackage[numbers]{natbib}

\theoremstyle{plain}
\newtheorem{thm}{Theorem}

\newtheorem{lem}[thm]{Lemma}
\newtheorem{prop}[thm]{Proposition}
\theoremstyle{definition}
\newtheorem{definition}[thm]{Definition}
\newtheorem{remark}[thm]{Remark}
\newtheorem{example}[thm]{Example}

\numberwithin{thm}{section}
\numberwithin{equation}{section}

\newcommand{\EQ}[1]{\eqref{#1}}

\newcommand{\THM}[1]{Theorem~\ref{#1}}

\newcounter{hypo}

\DeclareMathOperator*{\osc}{osc}
\DeclareMathOperator{\diam}{diam}
\DeclareMathOperator{\dist}{dist}

\DeclareMathOperator{\intr}{int}
\DeclareMathOperator{\CCA}{CCA}

\DeclareMathOperator{\USC}{USC}
\DeclareMathOperator{\LSC}{LSC}
\DeclareMathOperator*{\esssup}{ess\,sup}

\newcommand{\R}{\ensuremath{\mathbb{R}}}
\newcommand{\rn}{\R^n}

\newcommand{\ep}{\varepsilon}

\newcommand{\mrbox}[1]{\quad \mbox{#1} \ }

\newcommand{\bu}{\partial U} 
\newcommand{\bv}{\partial V}

\newcommand{\ra}{\rightarrow}
\newcommand{\beq}[1]{\begin{equation}\label{#1}}
\newcommand{\eeq}{\end{equation}}
\newcommand{\beqs}{\begin{equation*}}
\newcommand{\eeqs}{\end{equation*}}
\newcommand{\set}[1]{\left\{#1\right\}}
\newcommand{\is}[1]{\text{\ #1\ }}
\newcommand{\iq}[1]{\text{\quad #1\quad }}
\newcommand{\uc}{\bar U}\newcommand{\vc}{\bar V}
\newcommand{\lloc}{{\rm Lip_{loc}}}
\newcommand{\lp}{\left(}\newcommand{\rp}{\right)}
\newcommand{\la}[2]{L\lp\frac{#1}{#2}\rp}
\newcommand{\ca}{{\mathcal A}}
\newcommand{\ccas}{satisfies comparisons with cones from above}
\newcommand{\cca}{satisfy comparisons with cones from above}
\newcommand{\tend}{\bibliographystyle{plain}\bibliography{ccrituniq}\end{document}}
\newcommand{\cp}{{\mathcal P}}
\newcommand{\f}{f}
\newcommand{\g}{g}

\newcommand{\flow}{T}

\begin{document}
\title[Convexity criteria and uniqueness]{Convexity criteria and uniqueness of absolutely minimizing functions}
\author[S. N. Armstrong]{Scott N. Armstrong}
\address{Department of Mathematics\\ Louisiana State University\\ Baton Rouge, LA 70803.}
\email{armstrong@math.lsu.edu}
\author[M. G. Crandall]{Michael G. Crandall}
\address{Department of Mathematics\\ University of California, Santa Barbara \\ Santa Barbara, CA, 93106}
\email{crandall@math.ucsb.edu}
\author[V. Julin]{Vesa Julin}
\address{Department of Mathematics and Statistics\\ P.O. Box 35, FIN-40014\\ University of Jyv\"askyl\"a, Finland}
\email{vesa.julin@jyu.fi}
\author[C. K. Smart]{Charles K. Smart}
\address{Department of Mathematics\\ University of California\\ Berkeley, CA 94720.}
\email{smart@math.berkeley.edu}
\date{\today}

\begin{abstract}
We show that absolutely minimizing functions relative to a convex Hamiltonian $H:\rn\ra\R$ are uniquely determined by their boundary values under minimal assumptions on $H.$  Along the way, we extend the known equivalences between comparison with cones, convexity criteria, and absolutely minimizing properties,  to this generality. These results perfect a long development in the uniqueness/existence theory of the archetypal problem of the calculus of variations in $L^\infty.$
\end{abstract}\maketitle


\section{Introduction}\label{intro}

\maketitle


Consider a Hamiltonian $H: \rn\ra \R$ which satisfies
\beq{ch}
\left\{\begin{aligned}  & {\rm (i)} \ H  \is{is convex,}\\
 & {\rm (ii)} \ H(0)=\min_{\rn}H=0,\\
& {\rm (iii)} \is{the set} \set{p: H(p)=0} \is{is bounded and has empty interior.}
\end{aligned}\right.
\eeq 
Let $U\subset\rn$ be bounded and open, take $g\in C(\bu),$ and consider the archetypal problem of the calculus of variations in $L^\infty:$ find a function
\beqs
u\in\mathcal{A}:= \{ v \in C(\uc) : v = g \ \mbox{on}  \  \bu \}
\eeqs
such that
\beq{min} 
\|H(Du)\|_{L^\infty(U)}\le \|H(Dv)\|_{L^\infty(U)}\quad \mbox{for all} \ v\in \mathcal{A}.
\eeq
Here $Dv$ is the gradient of $v,$ and the quantity $\|H(Dv)\|_{L^\infty(U)}$ is assigned the value $\infty$ unless $v$ is locally Lipschitz continuous in $U$ (we write $v\in\lloc (U)$ for short), in which case $\| H(Dv) \|_{L^\infty(U)}$ may be defined in the usual way since $v$ is differentiable almost everywhere in $U$ by Rademacher's theorem.

Results of Barron, Jensen and Wang \cite{BJW1} and Champion, De Pascale and Prinari \cite{CDP} show that this problem has solutions provided that $g$ is the restriction to $\bu$ of some function $v\in C(\uc)\cap\lloc(U)$ for which $H(Dv)$ is essentially bounded. Structure conditions in \cite{BJW1} arising from the method used therein do not allow all the $H$'s satisfying \EQ{ch}, but the Perron method does produce minimizers at the full generality of \EQ{ch}, as shown in \cite{CDP}. However, even if $H(p)=|p|$ is the Euclidean norm, the minimizing functions are badly nonunique in general.  

Aronsson \cite{A1, A2} attacked this nonuniqueness phenomenon by adding conditions which included that \EQ{min} should not only hold on $U,$ but on each open subdomain $V$ of $U$ for which $\vc\subset U$ (which we will write as $V\ll U$): that is, if $V\ll U,$ then 
\beq{am} 
\|H(Du)\|_{L^\infty(V)}\le \|H(Dv)\|_{L^\infty(V)} \quad \mbox{for each}\   v\in C(\vc) \ \mbox{with} \ v=u\is{on} \bv. 
\eeq
We say that \emph{$u$ is absolutely minimizing for $H$ on $U$} if \EQ{am} holds for all $V\ll U.$ Notice that ``absolutely minimizing," as we have formulated it, {\it does not} by definition include that $u$ is minimizing, i.e., that \EQ{min} holds. In particular, it does not require $u$ to be defined on $\bu.$ The results of \cite{BJW1, CDP} provide absolutely minimizing solutions of \EQ{min} in a generality in which $H(Du)$ is replaced by $H(x,u,Du),$ and $H(x,s,p)$ only needs to be ``quasi-convex" or ``level-set convex" in $p,$ and so on, under suitable structure assumptions. However, the issue is uniqueness. Yu \cite{Yu2} noted that absolute minimizers are not uniquely determined by their boundary data even for $H$'s of the simple form $H=H(x,p)=|p|^2+f(x)$ and $n=1.$ 

\medskip

In this paper we perfect the theory of uniqueness for the case ``$H(p)$", and thereby complete a long story.  A special case of  results herein is that if $H$ satisfies \EQ{ch}, then the comparison principle holds; that is, if $u,v\in C(\uc)$ are absolutely  minimizing in $U$ for $H,$ then  
\beq{cth}
u(x)-v(x)\le \max_{y\in \bu} (u(y)-v(y)) \quad \is{for all} x\in U.
\eeq
It is easy to see that the requirement  of \EQ{ch}(iii) that the interior of $H^{-1}(0)$ be empty is a necessary condition for the comparison assertion \EQ{cth} to hold.  

It follows from this and the existence results already mentioned that there is exactly one function $u\in C(\uc)$ which is absolutely minimizing in $U$ and satisfies $u=g$ on $\bu.$     This is true even if $g$ is not the restriction to $\bu$ of any function $v$ for which the right hand side of \EQ{min} is finite, so the existence assertion goes a bit beyond the setting of \cite{CDP}.  However, it follows in a standard way from the continuity of absolutely minimizing functions in their boundary values provided by \EQ{cth}.   In all cases, the unique absolute minimizer $u$ satisfies \EQ{min}.

\medskip

What is truly new here is that $H$ is not assumed to be twice continuously differentiable, nor is it  assumed to be a norm.   Jensen, Wang and Yu \cite{JWY} established \EQ{cth} if $H$ is $C^2.$    More closely related to what we do herein is the case in which $H(p)=\|p\|$ is given by a norm $\|\cdot\|.$ While such $H$ are not necessarily $C^1,$  Armstrong and Smart \cite{AS} proved \EQ{cth} in this case (see also \cite{AS2}). Their proof relied on convexity/concavity properties of the functions 
\beq{cvcc}\left\{\begin{aligned}
&{\rm (i)} \ t\mapsto T^tu(x):= \max_{\|y-x\|_*\le t}u(y),\\
&{\rm (ii)} \ t\mapsto T_tv(x):=\min_{\|y-x\|_*\le t}v(y),
\end{aligned}\right.
\eeq
for $u$ and $v$ which are absolutely minimizing for $H(p) = \| p \|.$  Here 
\beqs
\|p\|_*:=\max_{\|q\|\le 1}p\cdot q,
\eeqs
is the dual norm to $\| \cdot \|,$ and $p\cdot q$ is the Euclidean inner-product of $p$ and $q.$ It is proved in Aronsson, Crandall and Juutinen \cite{tour} that the convexity of \EQ{cvcc}(i) is equivalent to $u$ being ``absolutely subminimizing" in $U$ for norms $H(p)=\|p\|$; a similar statement holds regarding the concavity of \EQ{cvcc}(ii). See Section \ref{prelim} for the definition of \emph{absolutely subminimizing}.

In \cite{AS} it is shown that if the map \EQ{cvcc}(i) (respectively, \EQ{cvcc}(ii)) is convex (respectively, concave) for $x\in U,$ and  
 \beq{ps}
0\le t< r(x):=\inf\set{\|y-x\|_*: y\in\bu},
\eeq
then \EQ{cth} holds. Before we outline further how we extend the method of \cite{AS} to handle uniqueness for more general convex Hamiltonians, we recall the role of partial differential equations in earlier works. This will allow us to paint the picture of the nature of the departure of \cite{AS,AS2} and the results herein from preceding works. 

As we show in Appendix \ref{viscsol}, if $u$ is absolutely minimizing for $H,$ then $u$ is a viscosity solution of the corresponding Aronsson equation, namely
\begin{equation} \label{ae}
\ca_H[u]:=H_{p_i}(Du) H_{p_j}(Du) u_{x_ix_j} = 0.
\end{equation}
Of course, $H$ need not be differentiable in our generality, and so the term $``H_{p_i}$" above must be interpreted in terms of the subdifferentials $\partial H(p).$ Recall that
\beqs
y_0\in\partial H(p_0) \iff H(p)\ge H(p_0)+(p-p_0)\cdot y_0\is{for all} p\in\rn.
\eeqs
The possible ``multi-valuedness" of these subdifferentials must be accommodated in interpreting \EQ{ae}.  
To say that $u$ is a \emph{viscosity subsolution} of \EQ{ae} means that $u$ is upper semicontinuous and if $u-\varphi$ has a local maximum at $x_0$ for some smooth $\varphi,$ then 
\beq{gae}
\max_{\omega\in\partial H(D\varphi(x_0))} \omega\cdot D^2\varphi(x_0)\omega\ge 0.
\eeq
Note that if $H$ is not $C^1,$ then the expression on the left of \EQ{gae} depends {\it discontinuously} on $D\varphi(x_0).$ \emph{Viscosity supersolutions} are defined in an analogous way, and a \emph{viscosity solution} is a function which is both a viscosity subsolution and a viscosity supersolution.  

The Hamiltonians
 \beq{uae} 
 H_1(p):=|p|   \quad   \mbox{and}   \quad   H_2(p):=\frac 12|p|^2,
 \eeq
where $|\cdot|$ is the Euclidean norm, define the same class of absolutely minimizing functions and lead to the Aronsson equations
 \beq{euli}
 \ca_{H_1} [u] =  \frac{\Delta_\infty u}{|Du|^2}=0   \quad   \mbox{and} \quad \ca_{H_2}[u]=\Delta_\infty u=0,
 \eeq
where 
\beq{lapinf}
\Delta_\infty u : = \sum u_{x_i}u_{x_j}u_{x_ix_j}
\eeq
is often called the \emph{infinity Laplacian}. If so, then $\ca_{H_1}$ is called the \emph{normalized} or \emph{1-homogeneous infinity Laplacian}. It is not hard to show that the viscosity subsolutions, defined as above, of  $\ca_{H_1}[u]=0$ are the same as the viscosity subsolutions of  $\ca_{H_2}[u]=0.$ It was proved by Jensen in \cite{RJ} that the absolutely minimizing functions for $H_1$ and $H_2$ are precisely the \emph{infinity harmonic functions}, that is, the viscosity solutions of $\Delta_\infty u=0.$ This result was extended to $H\in C^2$ by Gariepy, Wang and Yu \cite{GWY}.

\medskip

At the time of \cite{JWY}, the issue of the uniqueness of absolutely minimizing functions satisfying a given Dirichlet condition was still confounded with the issue of uniqueness of solutions of the Aronsson equation satisfying a Dirichlet condition. Jensen proved uniqueness of infinity harmonic functions satisfying a Dirichlet condition in the seminal paper \cite{RJ}, and this was extended to $H\in C^2$ in \cite{JWY}. As the absolutely minimizing property was characterized by the Aronsson equation for $H\in C^2$ owing to \cite{GWY}, the uniqueness of absolutely minimizing functions followed.

This was the situation until Peres, Schramm, Sheffield, and Wilson \cite{PSSW} discovered a new proof of uniqueness, which was noteworthy in that it did not use the Aronsson equation and held in greater generality. It relied on some complicated probabilistic arguments and a beautiful connection between the infinity Laplace equation and random-turn, two-player games. Finally, the uniqueness theory culminated in the the proof in \cite{AS}, which is elegant, elementary, and makes no use of partial differential equations (or probabilistic methods) and has no need for the viscosity solution machinery developed for second-order elliptic equations.

\medskip

We return to describing  the setting of the current work, in which convexity criteria, rather than the Aronsson equation, play the leading role. Indeed, it is not known whether the result of \cite{GWY} on the sufficiency of the Aronsson equation for the absolutely minimizing property can be extended to the generality of our work; this is certainly an outstanding open problem in the subject.   Until the latter question is resolved in the affirmative, PDE methods cannot be used to obtain the uniqueness result we prove in this work.

One may wonder what becomes of these ``convexity criteria" for more a general convex Hamiltonian $H.$ It was conjectured in Barron, Evans and Jensen \cite{BEJ} that the map \EQ{cvcc}(i) should be generalized to
\beqs
T^tu(x):=w(t,x),
\eeqs
where $w=w(t,x)$ is the viscosity solution of the Hamilton-Jacobi initial value problem
\beq{HJeq}
\left\{ \begin{aligned}
& w_t-H(Dw)=0 & \mbox{in} & \ (0,\infty) \times U, \\
& w= u & \mbox{in} & \ \{ 0 \} \times U.
\end{aligned} \right.
\eeq
The paper \cite{BEJ} was couched in the language of the PDEs involved, as was this conjecture, while in this introduction we are now replacing the PDEs by absolutely minimizing properties from which they can be derived.

Appropriate care has to be taken with exactly where the function $w$ above is well defined, and so forth. Put another way, to uniquely specify $w=w(t,x),$ we must specify boundary conditions on $(0,\infty) \times \bu.$ However, the Hamilton-Jacobi equation \EQ{HJeq} propagates disturbances at finite speeds, and for our purposes we are only concerned with the behavior of $w(t,x)$ for very small $t>0.$ In any case, we sidestep these issues by replacing the Hamilton-Jacobi equations by the Hopf-Lax formulas which they suggest, and refer no more to \EQ{HJeq} or other PDEs. That is, we will simply define
\beqs
\flow^tu(x):=\sup_{y\in U}\lp u(y)-t\la {y-x}t\rp
\eeqs
where $L$ is the Lagrangian of $H,$ defined by the formula
\beqs
L(y):=\sup_{p\in\rn}(y\cdot p-H(p)).
\eeqs
Along the way, we must pay attention to the $(t,x)$ regions in which various properties can be established.   

Note that if $H(p)=\|p\|,$ then its Lagrangian is the function
\beqs
L(y)=\begin{cases} 0 &\is{if} \|y\|_*\le1, \\
\infty &\is{if} \|y\|_*>1,\end{cases}
\eeqs
which gives rise to the formulas \EQ{cvcc}.  

For a clear discussion of the conjecture of \cite{BEJ} in the setting of Aronsson equations, we refer to Juutinen and Saksmann \cite{JS}, who confirmed that under the additional regularity assumption $H\in C^2(\R^n \setminus \{ 0 \}) \cap C^1(\R^n)$ and some further technical conditions, the convexity of the map $t\mapsto \flow^tu(x),$ on a suitable interval analogous to \EQ{ps}, is equivalent to $u$ being a viscosity subsolution of the Aronsson equation.

\medskip

In this work, we remove the assumption that $H\in C^2$ and verify directly that the convexity $t\mapsto \flow^tu(x)$ is equivalent to $u$ being absolutely subminimizing (see \THM{equivalences} below). This allows us to generalize the uniqueness proof of \cite{AS} to a Hamiltonian $H$ satisfying only \EQ{ch}. However, this cannot be done without further ado, owing to the more complex character of $L$ (or equivalently, $T^tu(x)$) in general. To make this point more clearly, we recall that $H_1$ and $H_2$ in \EQ{uae} have the same absolutely minimizing functions, and the corresponding Aronsson equations \EQ{euli} are equivalent in the viscosity sense. However, the Lagrangian for $H_2$ is 
\beqs
L_2(y)=\frac 12 |y|^2,
\eeqs
and its associated convexity criteria do not submit to the analysis of \cite{AS} so easily, in contrast to the case corresponding to $H_1.$

To overcome these subtleties, we use a ``patching" procedure which, roughly speaking, produces approximations of $u$ and $v$ which have gradients bounded away from 0. When coupled with an argument born in the work of Le Gruyer and Archer \cite{LeGA}, this permits us to carry out the schema of \cite{AS}. This patching procedure is not new, special cases having been used in uniqueness proofs in Crandall, Gunnarsson and Wang  \cite{CGW} when $H$ is a norm, and  in \cite{JWY} when $H$ is smooth. It is evidently a necessary step in the proof of Theorem \ref{compth}, as it is the only place in our arguments where the condition that $H^{-1}(0) = \set{p: H(p)=0}$ has empty interior, as required by \EQ{ch}(iii), is needed. In other places in this text, we can get by with only the assumption that $H^{-1}(0)$ is bounded.
 
\medskip

This paper is organized as follows. In Section \ref{prelim} we first set some notation and definitions and formulate our premier result, the comparison theorem described above.  The proof of the comparison theorem is presented in outline in Section \ref{outl}.  By ``in outline" we mean that the key ingredients of the argument are stated, but some of their proofs are deferred.  The deferred proofs make up most of Sections \ref{elrem}-\ref{cb}, and Section \ref{amtctc}.  Section \ref{patch} contains an approximation theorem which is also needed to prove the comparison result; Section \ref{end} contains a variety of tools used in Section \ref{patch} as well as the proof that the convexity criterion implies absolutely subminimizing. Appendix A contains the derivation of the Aronsson equation in the generality under discussion from the convexity criterion.

As our approach requires us to reorganize and generalize much of
the theory of absolutely minimizing functions from scratch, we were able
to make our presentation self-contained at little additional cost.
Therefore this paper is quite accessible to nonexpert readers.

\section{Notation, Preliminaries and the Main Result}\label{prelim}
The following conventions and assumptions are in force throughout this work.  For the reader's benefit, we repeat some material from the introduction in order to reposition it more conveniently.   

The Hamiltonian $H:\R^n\ra\R$ satisfies \EQ{ch}. That is, $H$ is convex, $H(0)=\min_{\R^n}H=0,$ and the level set $H^{-1}(0)$ is bounded and has empty interior. The symbols $U, V, W$ always denote open subsets of $\R^n.$ The closure and boundary of $U$ are denoted by, respectively, $\uc$ and $\bu.$  The notation $V\ll U$ means that $\vc$ is a compact subset of $U.$ Note that the set $U$ it allowed to be unbounded below, unless otherwise said.  

For $r>0,$ we put
\beq{dur}
U_r:=\set{x\in U: \dist(x,\bu)>r},
\eeq
where
\beqs
\dist(x,\bu):=\inf_{y\in\bu}|x-y|,
\eeqs
and $|x|$ denotes the Euclidean length of $x.$ In this regard, the ambiguous notation $\overline U_r$ signifies
\beq{urc}
\overline U_r:={\rm\ closure\ of\ }U_r.
\eeq
Open and closed balls are defined via the Euclidean distance and written
\beqs
B(x,r):=\set{y:|y-x|<r} \quad \mbox{and} \quad \bar B(x,r):=\set{y:|y-x|\le r}.
\eeqs
If $x, y\in\R^n,$ we denote the various line segments with endpoints $x$ and $y$ by
\beqs
\begin{split}
&[x,y]:=\set{(1-t)x+ty: 0\le t\le 1},\ (x,y):=\set{(1-t)x+ty: 0<t< 1},\\& [x,y):=\set{(1-t)x+ty: 0\le t< 1},
\end{split}
\eeqs
and so on.

The set $\lloc(U)$ consists of those functions $v:U\ra\R$ which are Lipschitz continuous on every $V \ll U.$ Recall that if $v\in\lloc(U),$ then $v$ is differentiable at almost every point in $U$ by Rademacher's theorem. If $v\in\lloc (U),$ then $Dv$ denotes its almost everywhere defined gradient, and 
\beq{HDv}
\|H(Dv)\|_{L^\infty(U)}:=\esssup_{x\in U}H(Dv(x)).
\eeq
If $v: U \to \R$ and $v \not\in \lloc(U),$ then the quantity on the left side of \EQ{HDv} is taken to be $\infty.$ Of course, this quantity may be $\infty$ even if $v \in \lloc(U).$

We say $u:U\ra \R$ is \emph{absolutely minimizing for $H$ in $U$} provided that $u\in\lloc(U)$ and \EQ{am} holds; that is, if \beq{cov}
 V\ll U \quad  \mbox{and}  \  v\in\lloc(V)\cap C(\vc) \quad \mbox{satisfies} \  v=u\is{on}\bv,
 \eeq
 then
  \beq{dam}
\|H(Du)\|_{L^\infty(V)}\le \|H(Dv)\|_{L^\infty(V)}.
\eeq
  We will prove that if $u, v\in C(\uc)$ are absolutely minimizing in $U$ and $U$ is bounded, then \EQ{cth} holds, that is
  \beq{ceqa}
u(x)-v(x)\le \max_{y\in \bu} (u(y)-v(y)) \quad \mbox{for every}  \  x\in U.
\eeq
 In fact, in Definition \ref{sam} immediately below, we split the notion of absolutely minimizing into two parts, which we call ``absolutely subminimizing" and ``absolutely superminimizing." In these terms, our premier result is:
\begin{thm}[Comparison Theorem]\label{compth} Let $U$ be bounded, $u,v\in C(\uc),$ $u$ be absolutely subminimizing for $H$ in $U$ and $v$ be absolutely superminimizing for $H$ in $U.$  Then \EQ{ceqa} holds. 
\end{thm}
  The requisite notions are:
\begin{definition}\label{sam}
A function $u \in \lloc(U)$ is called \emph{absolutely subminimizing} (respectively, \emph{absolutely superminimizing)  in $U$ with respect to $H$}, provided that \EQ{cov} and $u\ge v$ in $V$  (respectively, $u\le v$ in $V$) imply that  \EQ{dam} holds. 
\end{definition}

\subsection{The Proof of Theorem \ref{compth}} \label{outl}

The path we take through the proof of the comparison theorem involves extending much of the machinery used under other assumptions to the case under consideration, adding some new elements,  and combining it all just so. Rather than getting lost in the forest before bursting into the light near the end, let us sketch the skeleton of our arguments here. Flesh is added to the bones in the rest of this paper in the form of the proofs of results merely formulated in this section. 

A primary role in the proof will be played by the Hamilton-Jacobi flows mentioned in the introduction.  For our purposes, we take 
any $u: U \to \R,$ $x \in U,$ and $t> 0,$ and simply  define
\begin{equation} \label{defTt}
\flow^t u(x) : = \sup_{y\in U} \left( u(y) - t L\left( \frac{y-x}{t} \right) \right)
\end{equation}
and
\begin{equation} \label{defTtb}
\flow_t u(x) : = \inf_{y\in U} \left( u(y) + t L\left( \frac{x-y}{t} \right) \right).
\end{equation}
We also put 
\beq{flowz}
\flow^0u(x) :=\flow_0u(x):= u(x) 
\eeq
for $x\in U,$ which makes the maps $t\mapsto \flow^tu(x)$ and $t\mapsto \flow_t u(x)$ continuous from the right at $t=0$ (see Remark \ref{stcont}). We repeat that the Lagrangian $L: \R^n \to \R \cup \{ + \infty \}$ is the extended real-valued function defined by 
\begin{equation*}
L(q) := \sup_{p \in \R^n} \left( p\cdot q - H(p) \right):=\sup_{p \in \R^n} \left( \sum_{j=1}^n p_jq_j - H(p) \right).
\end{equation*}
Of course,
\begin{equation} \label{lzero}
L(0) = 0 \leq L(q) \quad \mbox{for every} \ q \in \R^n,
\end{equation}
 which follows from the definition of $L$ and $H(0)=\min_{\R^n}H=0,$ and this implies that 
\begin{equation} \label{flowup}
\inf_U u\le \flow_t u(x) \leq u(x) \leq \flow^t u(x) \le\sup_U u\quad \mbox{for all} \ x \in U, \  t > 0.
\end{equation}
In particular, $\flow^t u$ and $ \flow_t u$ are well defined and bounded if $u$ is bounded, as we will typically assume.  Also note that both $\flow^t$ and $ \flow_t$ are order preserving and commute with constants, that is, they have the properties of $S$ formulated below:
\beq{opcc}
u\le v\implies Su\le Sv, \is{and} S(u+c)=Su+c\is{for any constant} c. 
\eeq
Finally note that, as  defined above, the operators $\flow^t$ and $\flow_t$ depend on the underlying domain $U,$ although our notation does not explicitly display this dependence.

The following result  concerning how the absolutely subminimizing property of a function $u$ is reflected in the map $t\mapsto T^tu(x)$ is an important  component of the proof of Theorem \ref{compth}. 

\begin{prop}\label{ccrith} Let $u\in\lloc(U)$ be bounded and absolutely subminimizing with respect to $H$ in $U,$ and let $V\ll U.$  Then there exists $\delta>0$ such that \beq{ccri}
\mbox{the map} \ t\mapsto \flow^tu(x)\quad \mbox{is convex on the interval} \ [0,\delta] \quad \mbox{for each} \ x\in V. 
\eeq
\end{prop}
For convenience, we give the property defined by the conclusion of this theorem a name, along with a name for the corresponding property enjoyed by absolutely superminimizing functions.

\begin{definition}\label{dccri} A bounded function $u\in  C(U)$ \emph{satisfies the convexity criterion in $U$} provided that for every $V\ll U$ there exists $\delta>0$ such that \EQ{ccri} holds. Likewise, we say that $u$ \emph{satisfies the concavity criterion in $U$} if for every $V\ll U$ there exists $\delta>0$ such that 
\beq{cacri}
\mbox{the map} \ t\mapsto \flow_tu(x)\  \mbox{is concave on the interval} \ [0,\delta] \  \mbox{for each} \ x\in V.  
\eeq
\end{definition}
\begin{remark} 
\label{remcvxcav} The reader will note that Proposition \ref{ccrith} makes no assertion relating  ``absolutely superminimizing" and the concavity criterion. This is designed to highlight the current remark to the effect that it is not necessary to do so.  Set $\hat H(p):=H(-p)$ and observe that $\hat H$ satisfies \EQ{ch} in place of $H.$  Then, by the definitions, $u$ is absolutely superminimizing for $H$ in $U$ if and only if $-u$ is absolutely subminimizing for $\hat H$ in $U.$  The Lagrangian $\hat L$ of $\hat H$ is given by $\hat L(q)=L(-q).$  It  follows from Proposition \ref{ccrith} that if $u$ is absolutely superminimizing for $H$ in $U,$ then $-u$ satisfies the convexity criterion generated by $\hat L,$ which is the assertion that for every $V\ll U$ there is a $\delta>0$ such that the map
\beqs
t\mapsto \sup_{y\in U}\lp -u(y)-t \hat L\lp\frac{y-x}t\rp\rp=-\inf_{y\in U}\lp u(y)+t  L\lp\frac{x-y}t\rp\rp
\eeqs
is convex on $[0,\delta]$ for each $x\in V.$ This is equivalent to the statement that the map $t\mapsto \flow_t u(x)$ is concave on the same interval. 
\end{remark}  

Proposition \ref{ccrith} follows immediately from the initial results of Section \ref{amtctc}.    The  following result, which has to do with domain of dependence and speed of propagation issues for the Hamilton-Jacobi flows, is called upon often along the way. For $u:U\to \R,$ we use the notation 
\beq{dosc}
\osc_U u := \sup_Uu -\inf_Uu.
\eeq

\begin{lem} \label{local}  Let $u:U\ra\R$ be upper semicontinuous and bounded. Then
for every $\alpha, r > 0,$ there exists $t_0 = t_0 (\alpha,r) > 0$ such that for all $0 < t < t_0$ and $x\in U,$
\begin{equation} \label{elocal}
\flow^t u(x) = \sup_{y\in  B(x,r) \cap U} \left( u(y) - t L\left( \frac{y-x}{t} \right) \right)\!,
\end{equation}
provided that $\osc_U u \leq \alpha.$ Moreover, if $t,s>0$ satisfy $t+s < t_0(\alpha,r)$ and $\dist(x,\partial U) > r,$ then
\begin{equation} \label{uepadd}
\flow^{t+s}u(x) = \flow^t \left( \flow^s u\right)(x).
\end{equation}
\end{lem}

We prove Lemma \ref{local} in the following subsection.  

In addition to the identity \EQ{uepadd}, valid for sufficiently small $t,s>0,$ we recall the properties \EQ{flowup}, valid for any $t>0.$  Also, directly from the definitions, we have  
\begin{equation} \label{turns}
\flow^t\left( \flow_t u\right) (x) \leq u(x) \leq \flow_t\left( \flow^t u\right) (x)\quad \mbox{for all} \ x \in U, \  t > 0.
\end{equation}

Now suppose $U$ is bounded and that $u, v\in C(\uc)$ are, respectively, absolutely subminimizing and absolutely superminimizing for $H$ in $U$ and $\osc_U u, \osc_Uv\le \alpha.$ Let $r>0$ (it will be sent to zero shortly). 
By Proposition \ref{ccrith}, the map $t \mapsto T^tu(x)$ is convex and $t \mapsto T_t v(x)$ is concave on $[0, \delta)$ for every $x \in U_{2r}$ for some $\delta >0.$   Select $0 < t <  \min\left\{ t_0(\alpha,r/2), \delta \right\}/2$ and denote $u^t (x):= \flow^t u(x)$ and $v_t(x) := \flow_t v(x).$ Using \EQ{turns}, \EQ{uepadd}, the convexity criterion for $u$ and the concavity criterion for $v,$ we obtain
\begin{equation*}
\flow^t u^t(x) + \flow_t u^t(x) - 2u^t(x) \geq \flow^{2t} u(x) + u(x) - 2 \flow^t u(x) \geq 0,
\end{equation*}
as well as
\begin{equation*}
\flow^t v_t(x) + \flow_t v_t(x) - 2v_t(x) \leq v(x) + \flow_{2t}v(x) - 2 \flow_t v(x) \leq 0.
\end{equation*}

That is, $u^t=\flow^t u$ and $v_t=\flow_tv$ satisfy the hypotheses placed on $\f$ and 
$\g$  in the following lemma. See also Remark \ref{stinx}.

\begin{lem}[Stationary point lemma] \label{SPlem}
Assume that $U$ is bounded,  $\alpha, r > 0,$ and $\f,\g\in C(\bar U_r ),$ $\osc_{U} \f, \osc_{U} \g \leq \alpha,$ and that for some $0 < t < t_0(\alpha,r),$
\begin{equation}\label{findiff}
\flow^t \f(x) + \flow_t \f(x) - 2\f(x) \geq 0 \geq \flow^t \g(x) + \flow_t \g(x) - 2\g(x)   
\end{equation}
for every $x \in U_{2r}.$ Then either
\begin{equation}\label{maxok}
\max_{\bar U_r} (\f-\g) = \max_{\bar U_r\setminus U_{2r}} (\f-\g),
\end{equation}
or else there exists a point $x_0 \in U_{2r}$ such that
\begin{equation} \label{statpt}
\f(x_0) = \flow^t \f(x_0) = \flow_t \f(x_0) \quad  \mbox{and} \quad  \g(x_0) = \flow^t \g(x_0) = \flow_t \g(x_0).
\end{equation}
\end{lem}

Lemma \ref{SPlem} is proved in Section \ref{pspl}.

In Section  \ref{patch}, we reduce the proof of Theorem \ref{compth} to establishing it under the additional hypothesis 
\beq{ac}
S^+u(x) := \limsup_{s\downarrow 0}\frac{\flow^s u(x)-u(x)}{s}>0\quad \mbox{for every} \ x\in U.  
\eeq
Indeed, for $\gamma>0,$ Lemma \ref{patching} provides an absolutely subminimizing function $u_\gamma\le u$ which agrees with $u$ on $\bu,$ and for which the quantity on the left side of \EQ{ac} is at least $\gamma$ everywhere, and such that $u_\gamma\ra u$ as $\gamma\downarrow 0.$  
The convexity criterion implies that
\beqs
\frac{\flow^{t}u^t(x)-u^t(x)}t=\frac{\flow^{2t}u(x)-\flow^tu(x)}t\ge \frac{\flow^{t}u(x)-u(x)}t\ge \limsup_{s\downarrow 0}\frac{\flow^s u(x)-u(x)}{s}
\eeqs
for all $t>0$ in the interval of convexity. Thus, if \EQ{ac} holds, then there is no point $x_0$ satisfying $\flow^tu^t(x_0)=u^t(x_0)$ as in \EQ{statpt}, and we must therefore have the alternate possibility of \EQ{maxok}. That is, for all $t>0$ sufficiently small we have
\begin{equation}\label{maxoka}
\max_{\bar U_r} (u^t-v_t) = \max_{\bar U_r\setminus U_{2r}} (u^t-v_t).
\end{equation}
We then obtain \EQ{ceqa} upon sending $t\downarrow 0$ while invoking Remark \ref{stcont} below, and then sending $r\downarrow 0.$

At this point, we have proved Theorem \ref{compth}, pending proofs of the results cited in this discussion.
 
\subsection{Elementary Remarks About $H,$ $L,$ $\flow_t$ and $\flow^t$} \label{elrem}
 
In this subsection we review the basic facts we need regarding the relationship between $H$, $L$ and the Hamilton-Jacobi flows $T^t$ and $T_t$. We refer to Evans \cite{E} for more background, including a derivation of the Hopf-Lax formula for solutions of the Hamilton-Jacobi equation.

According to \EQ{ch}, the zero level set of $H$ is contained in the ball $B(0,R_0)$ for some $R_0 >0.$ Set
\beqs
k_0 :=\inf_{|p|= R_0}H(p)>0.
\eeqs
By the convexity of $H$ and $H(0)=0,$ we have
\beq{hcoer}
H(p)\ge \frac{k_0}{R_0} |p| \quad \mbox{for all} \ |p| \geq R_0.
\eeq
Similarly, if $H(p) >0,$ then the map
\beq{noflata}
t\mapsto H(tp) \quad \mbox{is strictly increasing on the interval} \  \left[\hat t,\infty\right),
\eeq
where $\hat t : =\sup\set{t\in[0,\infty): H(tp)=0}.$ In particular, for all $k\geq0,$ the level set
\beq{noflat}
H^{-1}(k):=\set{p:H(p)=k} \quad \mbox{has empty interior.} 
\eeq

The Lagrangian $L$ is obviously convex and we have already noted that it satisfies \EQ{lzero}. Observe that we may write $H$ in terms of $L$ as
\begin{equation} \label{HbyL}
H(p) = \sup_{q\in \R^n} \left( p\cdot q - L(q) \right).
\end{equation}
Indeed, by the definition of $L$ the right side of \EQ{HbyL} is equal to
\begin{equation*}
\sup_{q} \inf_{\tilde p} \left( q\cdot (p - \tilde p) + H(\tilde p) \right),
\end{equation*}
which, by considering the choices $\tilde p = p$ and $q\in \partial H(p),$ is seen to be equal to $H(p).$

In view of \EQ{hcoer}, if $|p|\ge R_0,$ then we have
\beqs
q\cdot p-H(p)\le \lp|q| -\frac{k_0}{R_0}\rp|p|.
\eeqs 
If $|q|< k_0/R_0,$ it follows that 
\beq{lfinite}
L(q)=\max_{|p|\le R_0}(q\cdot p-H(p))<\infty. 
\eeq
Observe also that for any $\alpha > 0,$ 
\begin{equation*}
\liminf_{|q| \to \infty} \frac{L(q)}{|q|} \geq \liminf_{|q|\to \infty} \left( \alpha - |q|^{-1} \max_{|p| \leq \alpha} H(p) \right) = \alpha.
\end{equation*}
Therefore, there is a function $M:[0,\infty)\ra [0,\infty)$ such that
\beq{pm}
M\ \mbox{is nondecreasing,}\quad \lim_{r\ra\infty} M(r)=\infty, \quad \mbox{and} \ L(q)\ge M(|q|)|q|. 
\eeq

\begin{proof}[{\bf Proof of Lemma  \ref{local}}]

According to \EQ{pm}, we may select $t_0 > 0$ so small that 
\begin{equation*}
M\lp\frac r{t_0}\rp> \frac\alpha r+1. 
\end{equation*}
Then for any $0 < t < t_0$ and $x \in U,$ we have, by the above and \EQ{pm}, 
\begin{equation*}
\begin{aligned}
\sup_{y \in U \setminus B(x,r)} \Big( u(y) - t L\left( \frac{y-x}{t} \right) \Big)& \leq \sup_U u - r M\lp\frac rt\rp\\ & <\sup_U u -\osc_Uu-r<\inf_U u\le u(x).
\end{aligned}
\end{equation*}
According to \EQ{flowup}, we deduce that \EQ{elocal} must hold.

Suppose now that $s,t> 0$ are such that $t+s < t_0(\alpha,r)$ and $\dist(x,\partial U) > r.$ In view of  the upper semicontinuity of $u$ and \EQ{elocal}, we may select $y \in \bar B(x,r)$ such that 
\begin{equation}\label{tps}
\flow^{t+s}u(x) = u(y) - (t+s) L\left( \frac{y-x}{t+s} \right).
\end{equation}
The first three lines below are valid for any $y, z\in U,$ while for  the fourth line we have put $z:=(ty+sx)/(t+s).$   We have
\begin{align*}  
\flow^t \left( \flow^s u\right)(x) & =  \sup_{\tilde z\in U}\lp \flow^su(\tilde z)-t\la{\tilde z-x}t\rp\\
&= \sup_{\tilde z \in U}\sup_{\tilde y\in U} \left( u(\tilde y) - sL\left( \frac{\tilde y - \tilde z}{s}\right) - t L\left( \frac{\tilde z - x}{t} \right) \right) \\
& \geq u(y) - s L\left( \frac{y - z}{s}\right) - t L\left( \frac{z - x}{t} \right) \\
& = u(y) - (t+s) L \left( \frac{y-x}{t+s} \right).
\end{align*}
From the above and \EQ{tps}, we conclude that
\beqs
\flow^t(\flow^su)(x)\ge \flow^{t+s}u(x). 
\eeqs
On the other hand, for all $\tilde y$ and $\tilde z$ we have
\beqs
(t+s)\la{\tilde y-x}{t+s}\le sL\left( \frac{\tilde y - \tilde z}{s}\right) + t L\left( \frac{\tilde z - x}{t} \right), 
\eeqs
because $L$ is convex and 
\beqs
 \frac{\tilde y-x}{t+s}= \frac s{t+s}    \frac{\tilde y - \tilde z}{s}  +\frac t{t+s} \frac{\tilde z - x}{t}.
\eeqs
It therefore follows from the second  expression for $\flow^t(\flow^su)(x)$ above that we always have
\beqs
\flow^t(\flow^su)(x)\le \flow^{t+s}u(x), 
\eeqs
completing the proof.
\end{proof}

\begin{remark}\label{indu} It was noted before that our definition of $\flow^t u(x)$ ``depends on $U$", the assumed domain of definition of $u.$  However, the proof just given shows that 
\beq{indep}
T^tu(x)=\max_{y\in \bar B(x,r)}\left( u(y)-t\la{x-y}t\right)
\eeq
whenever $t<t_0(\alpha,r)$ and $x\in U_r,$ no matter the choice of $U,$ provided $\osc_Uu\le\alpha.$   
\end{remark}
\begin{remark}\label{stcont} 
It follows from Lemma \ref{local} and \EQ{flowup} that if $u\in C(U)$ is bounded, then 
\beqs
 \lim_{t\downarrow 0}\flow^t u=\lim_{t\downarrow 0}\flow_t u=u
\eeqs
holds uniformly on compact subsets of $U.$ Indeed, fixing $r>0,$ for $t<t_0(\osc_Uu,r),$ 
\beqs
u(x)\le \flow^t u(x)\le \sup_{\bar B(x,r)\cap U} u.
\eeqs
If $u$ is bounded and uniformly continuous, then the convergence is uniform on $U.$ 
\end{remark}

\begin{remark}\label{stlip}
For later use, we refine Remark \ref{stcont} in the case that $u$ is locally Lipschitz continuous.  Suppose that $x\in U,$ $r>0,$ 
\beqs
|u(y)-u(x)|\le K|x-y|\iq{for} y\in B(x,r)\cap U, 
\eeqs
 and $\osc_Uu<\infty.$ According to \EQ{pm}, we may select $a> 0$ so large that $L(z) > K|z|$ for every $|z| > a.$ Then
\begin{equation*}
K|x-y| - t L\left( \frac{y-x}{t} \right) < 0 \quad \mbox{for all} \ |x-y| > at.
\end{equation*}
Therefore, if $0 < t < t_0(\osc_U u, r),$ then 
\begin{equation}\label{lipvarphi}
\begin{aligned}
\flow^t u(x) - u(x) & = \sup_{y\in B(x,r)\cap U} \left( u(y) - u(x) - t L\left( \frac{y-x}{t} \right) \right)  \\
& \leq \sup_{y \in B(x,r)} \left( K|y-x| - t L\left( \frac{y-x}{t} \right) \right) \\
& \leq \sup_{y \in B(x,at)} \left( K|y-x| - t L\left( \frac{y-x}{t} \right) \right) \\
& \leq (aK)t.
\end{aligned}
\end{equation}
\end{remark}

\begin{remark}\label{stinx}
We have not yet discussed any properties of the map 
\beq{xdep}
x\mapsto \flow^tu(x).
\eeq
We assume the notation of Lemma  \ref{local}. Suppose that $u \in  C(U)$ and $\osc_Uu\le\alpha.$ For $x\in U_{r}$ and $t<t_0(\alpha,r)$ we may write (\ref{xdep}) as
\beqs
x\mapsto\flow^t u(x)=\sup_{z \in B(0,r)}\lp u(z+x)-t\la{z}t\rp.
\eeqs
Suppose $x_0\in U_r$ and $\dist(x_0,\bu)=r_0>r.$  The maps $x \mapsto u(z+x),$  $z\in B(0,r),$ are equicontinuous in $x$ on $B(x_0,r_1)$ for any $0<r_1<r_0-r.$   Therefore   $\flow^t u \in C( B(x_0,r_1)),$ since it is the supremum of a family of equicontinuous functions on $ B(x_0,r_1)$ which are uniformly bounded above. As $x_0\in U_r$ is arbitrary, we deduce that $\flow^t u \in C( U_r).$    
\end{remark}

\subsection{Proof of the Stationary Point Lemma}
\label{pspl}

We note that  Lemma \ref{SPlem} is a generalization of \cite[Theorem 3.3]{LeGA}, which also appeared without proper attribution as \cite[Lemma 4]{AS}. Our proof follows along similar lines as the argument in \cite{LeGA}, although the situation is more complicated in our context of a general convex Hamiltonian $H.$ In particular, the conclusion is weaker and includes the extra alternative \EQ{statpt}, which is ruled out in our  proof of Theorem \ref{compth} as explained above. The   alternative \EQ{statpt} in the statement is essential; it is not true that \EQ{maxok} holds in general. This follows from the observation that the comparison theorem does not hold if the zero level set of $H$ has nonempty interior, and the empty interior assumption is utilized in our proofs only to rule out \EQ{statpt}  in applying the lemma to suitable approximations.

We also note that the lemma stands alone in the sense that its only use herein is in the proof of Theorem \ref{compth}. Thus this subsection may be deferred without compromising the reading of other parts of this paper. 

\begin{proof}[{\bf Proof of Lemma \ref{SPlem}}]
Consider the case that \EQ{maxok} fails. Then the set 
\begin{equation*}
E:= \left\{ x\in \bar U_r : (\f-\g)(x) = \max_{\bar U_r} (\f-\g) \right\}
\end{equation*}
is nonempty, closed, and contained in $U_{2r}.$ Define the set
\begin{equation*}
F:= \left\{ x \in E : \f(x) = \max_E \f \right\}
\end{equation*}
which is nonempty, closed, and contained in $E\subseteq U_{2r}.$ Consider a point $x_0 \in F.$ Since $x_0 \in E,$ we see that
\begin{equation} \label{xnot}
 \g(y) - \g(x_0) \geq \f(y)-\f(x_0) \quad \mbox{for every} \ y \in U_r.
\end{equation}
Using  \EQ{findiff},  \EQ{xnot} and \EQ{opcc}, and \EQ{findiff} again, in that order, together with $x_0\in U_{2r},$ and  $t < t_0(\alpha,r),$ we discover 
\begin{equation}\label{LGstring}
\begin{aligned}
\flow^t \f(x_0) - \f(x_0) \geq \f(x_0) - \flow_t \f(x_0)  \ge \g(x_0) - \flow_t \g(x_0) \geq \flow^t \g(x_0) - \g(x_0).
\end{aligned}
\end{equation}
Using \EQ{xnot} and \EQ{opcc} again, we also have $\flow^t \f(x_0) - \f(x_0) \leq \flow^t \g(x_0) - \g(x_0),$ and thus we must have equality in every inequality of \EQ{LGstring}. That is,
\begin{equation}\label{xnoteqs}
\begin{aligned}
\flow^t \f(x_0) - \f(x_0) =  \f(x_0) - \flow_t \f(x_0) = \g(x_0) - \flow_t \g(x_0) = \flow^t \g(x_0) - \g(x_0).
\end{aligned}
\end{equation}
Select a point $z\in \bar B(x_0,r) \subseteq U_r$ such that 
\begin{equation*}
\flow^t \f(x_0) = \f(z) - tL\left(\frac{z-x_0}{t} \right).
\end{equation*}
Observe that then, via \EQ{xnoteqs}, 
\begin{align*}
\g(z) & \leq tL\left( \frac{z-x_0}{t} \right) + \flow^t \g(x_0) \\
& = tL\left( \frac{z-x_0}{t} \right) + \flow^t \f(x_0) + \g(x_0) - \f(x_0) \\
& = \f(z) + \g(x_0) - \f(x_0).
\end{align*}
Therefore $\f(z) - \g(z) \geq \f(x_0) - \g(x_0)$ and thus $z \in E.$ Recalling that $x_0 \in F,$ we see that 
\begin{equation*}
\f(z) \leq \f(x_0) \leq \flow^t \f(x_0) = \f(z) - tL\left( \frac{z-x_0}{t} \right) \leq \f(z).  
\end{equation*}
Thus $\f(x_0) = \flow^t \f(x_0).$ Recalling \EQ{xnoteqs}, we obtain \EQ{statpt}.
\end{proof}

\subsection{Cone Basics}\label{cb}

We begin this subsection by noticing that \EQ{hcoer} and \EQ{noflata} imply that 
\beq{ph}
H^{-1}(k):=\set{p: H(p)=k}=\partial \set{p: H(p)<k},
\eeq
for every $k>0,$ and that $H^{-1}(k)$ is nonempty,  compact, and has empty interior. For $k \geq 0,$ we define the \emph{cone function}
\begin{equation*}
C_k (x) := \max \left\{ p\cdot x : H(p)=k \right\}\!.
\end{equation*}
It is evident that  $C_k$ is convex, positively homogeneous, subadditive,   Lipschitz continuous uniformly for bounded $k,$ and $C_k(x)>0$ for every $k>0$ and $x\not=0.$  

Furthermore, it is easy to show that the map $(x,k)\mapsto C_k(x)$ is continuous as well as strictly increasing in $k,$ by \EQ{noflata}. Setting
\beqs
M_k:=\min_{|x|=1}C_k(x),
\eeqs 
we observe that 
\beqs
C_k(x)\ge M_k|x|\quad \mbox{and} \quad M_k\ge r_k\quad \mbox{provided} \ B(0,r_k)\subseteq\set{H\le k}. 
\eeqs
Clearly the largest such $r_k\ra\infty$ as $k\ra \infty,$ and thus
\beq{ccoer}
C_k (x) \ge M_k|x|\is{where} M_k\ra\infty \is{as} k\ra\infty.
\eeq

Suppose now that $C_k$ is differentiable at $x$ and let $p\in H^{-1}(k)$ be such that $C_k(x)=p\cdot x.$ Then 
\beqs
C_k(x+y)\ge p\cdot(x+y)=C_k(x)+p\cdot y.
\eeqs
 It follows that $DC_k(x)=p,$ and then $H(DC_k(x))=H(p)=k.$  (This all is a special case of  standard remarks regarding the minimum of a family of $C^1$ functions at points of differentiability.) Hence 
\begin{equation}\label{hdck}
H\left( D C_k(x) \right) = k \quad \mbox{for almost every} \ x\in \R^n.
\end{equation}

\begin{remark}\label{gencone}  The cones above are built from $H,$ but we won't record their dependence on $H$ in our notation. Note that the cones $\hat C_k$ built from $\hat H,$ where $\hat H(p):=H(-p),$ are given by $\hat C_k(x)=C_k(-x),$ which in light of Remark \ref{remcvxcav} accounts for some sign changes in the statements below. 
\end{remark}

The property of comparison with cones, defined below, will play an intermediate role in the next section, as we will pass from ``absolutely (sub/super)minimizing" to ``comparisons with cones" to ``convexity criteria." 

\begin{definition}\label{dcc}
We say that a function $u\in \USC(U)$ (i.e., $u:U\ra\R$ is upper semicontinuous) satisfies \emph{comparisons with cones from above} in $U,$ if  
\begin{equation*}
\max_{x \in \bar V} \left( u(x) - C_k(x-x_0) \right) = \max_{x \in \bv } \left( u(x) - C_k(x-x_0) \right)
\end{equation*}
whenever
\beq{kvx}
k\geq 0,\ V\ll U \is{and} x_0 \in \R^n\setminus V.
\eeq
 Similarly, we say that $v\in \LSC(U)$ satisfies \emph{comparisons with cones from below} in $U,$ provided that
\begin{equation*}
\min_{x \in \bar V} \left( v(x) + C_k(x_0-x) \right) = \min_{x \in \bv  } \left( v(x) + C_k(x_0-x) \right),
\end{equation*}
whenever \EQ{kvx} holds.  We say that $u\in C(U)$ satisfies \emph{comparisons with  cones} in $U$ if it satisfies comparisons with cones both from above and below.  
\end{definition}

We now recall the standard argument which demonstrates that functions satisfying comparisons with cones from above are necessarily locally Lipschitz continuous. In the statement, we use the constants $M_k$ from \EQ{ccoer}, and also the Lipschitz constant of $C_k,$ that is, the least constant $K_k$ for which
\beq{ckl}
C_k(z)\le K_k|z|\iq{for} z\in\R^n.
\eeq

\begin{lem} \label{ccca-lip}
 If $u:U\ra \R$ is upper semicontinuous and \ccas\ in $U,$ then $u\in C(U).$ Moreover,  if $R>0,$ $B(x,R)\cup B(y,R)\ll U,$ $|x-y|\le R$ and $k$ is so large that
 \beq{lipa}
\frac{\osc_{B(x,R)\cup B(y,R)} u}{R}\le M_k,
 \eeq
 then 
 \beq{lipb}
 |u(x)-u(y)|\le K_k|x-y|.   
 \eeq
\end{lem}
\begin{proof}
We will first show that $u$ is lower semicontinuous-- and therefore continuous-- in $U.$ We must show that the set $E_l := \{ x\in U: u(x) > l \}$ is open for every $l\in \R.$ Fix $x\in E_l$ and put $\delta :=  \dist(x,\partial U)/3$ and $m:= \max_{\bar B(x,2\delta)} u.$ Select $k>0$ so large that
\begin{equation} \label{ccalipA}
m-l \leq \min_{|z|\geq \delta} C_k(z),
\end{equation}
and then select $0< \eta < \delta$ so  small that
\begin{equation} \label{ccalipB}
\max_{|z|\leq \eta} C_k(z) < u(x) - l.
\end{equation}
We claim that $B(x,\eta) \subseteq E_l.$ Suppose on the contrary that $y\in B(x,\eta)$ is such that $u(y) \leq l.$ Then using \EQ{ccalipA} we have
\begin{equation*}
u(z) \leq l + C_k(z-y) \quad \mbox{for} \ z\in \partial B(x,2\delta) \cup \{ y \} =\partial (B(x,2\delta)\setminus\set{y}).
\end{equation*}
Since $u$ \ccas, we deduce
\begin{equation*}
u(z) \leq l + C_k(z-y) \quad \mbox{for } \ z\in B(x,2\delta).
\end{equation*}
Inserting $z=x,$ we obtain a contradiction to \EQ{ccalipB}. Our claim that $E_{l}$ is open is confirmed. It follows that $u\in \LSC(U)$ and hence $u \in C(U).$

 Suppose now that the assumptions on $x, y, R, k$ are satisfied; in this regard, note that the left hand side of \EQ{lipa} is finite because $u$ is continuous.   If $|z-y|=R,$  then we have 
\beq{bige}
C_k(z-y)\ge M_k |z-y|\ge M_kR\ge \osc_{B(y,R)}u\ge u(z)-u(y). 
\eeq
Since $u$ satisfies comparisons with cones from above and \EQ{bige} also holds at $z=y,$ this inequality persists for all $|z-y|\le R;$ in particular,  it holds at $z=x$. Thus 
\beqs
u(x)-u(y)\le C_k(x-y)\le K_k|x-y|.
\eeqs  
Our assumptions are symmetric in $x$ and $y,$ and so we have established \EQ{lipb}.
\end{proof}
\begin{remark}\label{trueloc} To use the formulation of Lemma \ref{ccca-lip} to show that $u$ is Lipschitz continuous in some neighborhood of a given $x_0\in U,$ do as follows.  Suppose $R>0$ and $B(x_0,2R)\ll U.$ If $x,y\in B(x_0,R/2),$ then $B(x,R), B(y,R)\subset B(x_0,3R/2)\ll U$ and $|x-y|\le R.$ 
\end{remark} 
\begin{remark} \label{globalc} Suppose that $U=\R^n,$ and $u\in C(\R^n)$ is bounded and \ccas.  Then we may let $R\ra\infty,$ and satisfy \EQ{lipa} with $k\ra 0.$ This results in the information $u(x)-u(y)\le C_0(x-y)$ for $x, y\in\R^n.$   If $H^{-1}(0)=\set{0},$ then $C_0\equiv 0,$ and $u$ is constant.  Otherwise, we have $u(x)-u(y)\le \max_{H(p)=0}p\cdot(x-y).$ If $H^{-1}(0)\setminus \{ 0 \}$ is nonempty, then one can construct bounded, nonconstant functions which \cca \ in $\R^n.$
\end{remark}
\begin{remark}\label{onur} Suppose $u$ \ccas\ in $U$ and $\osc_Uu<\infty.$   If $R>0,$ $x,y\in U_{R},$ and $|x-y|\le R,$ then, applying the lemma,  \EQ{lipb} holds provided that $\osc_Uu/R\le M_k.$  On the other hand, if  $|x-y|\ge R, $ we have
\beqs
u(x)-u(y)\le \osc_U u=\frac{\osc_U u}R R\le \frac{\osc_U u}R|x-y|\le M_k|x-y|. 
\eeqs
Also note that $M_k|z|\le C_k(z)\le K_k|z|$ implies $M_k\le K_k.$  Hence 
\beq{lipd}
|u(x)-u(y)|\le A_R|x-y|\is{for} x,y\in U_R.
\eeq
where $A_R:=K_k.$  We have switched notation here to reflect the domain $U_R$ on which the Lipschitz condition holds. 
\end{remark}

This subsection concludes with a result establishing a connection between estimates of the form $u(x)-u(y)\le C_k(x-y),$ which appeared above, and estimates on $H(Du).$

\begin{lem}\label{htoc} Let $u\in\lloc(U)$ and $k\ge 0.$   Then the conditions
\begin{equation}\label{esth}  H(Du)\le k\is{ a.e. in} U
\end{equation}
and
\begin{equation}\label{optconest}
u(x)-u(y)\le C_k(x-y)\iq{provided that} [x,y]\subseteq U
\end{equation}
 are equivalent.
 \end{lem}
\begin{proof}  Let us assume \EQ{esth} and prove \EQ{optconest}. 
First mollify $u$ by defining
\begin{equation*}
u_\ep (z): = \int_{B(x,\ep)} \rho\left( \frac{z-y}{\ep} \right)  u(y) \, dy,
\end{equation*}
where $\rho \in C^\infty(\R^n)$ has support in $B(0,1),$ $\rho \geq
0,$ and $\int_{\R^n} \rho(y) \, dy = 1$. Then $u_\ep \in
C^\infty(U_\ep)$ for every $\ep > 0,$ and, by Jensen's inequality,
\begin{equation*}
H(Du_\ep(z)) \leq \int_{B(0,\ep)} \rho\left( \frac{z-y}{\ep}\right)
H(Du(y)) \, dy \leq k \quad \mbox{for each} \ z \in U_\ep.
\end{equation*}
It follows that
\begin{align*}
u_\ep(x) - u_\ep(y) & = \int_0^1 (x-y) \cdot Du_\ep(y + t (x-y)) \, dt \\
& \leq \int_0^1 \sup_{H(p) \leq k} (p \cdot (x-y)) \, dt = C_k(x-y),
\end{align*}
provided that $[x,y] \subseteq U_\ep$. Sending $\ep \downarrow 0,$ we
obtain \EQ{optconest}.

To prove the converse, we  assume \EQ{optconest} and that $u$ is differentiable at $y\in U.$  Set $x:=y+tz$ in \EQ{optconest}, divide by $t>0$ and let $t\downarrow 0$ to find that for all $z\in \R^n,$
\beqs
Du(y)\cdot z\le C_k(z)=\max_{p\in H^{-1}(k)} p\cdot z=\max_{p\in H^{-1}([0,k])} p\cdot z. 
\eeqs
As $H^{-1}([0,k])$ is closed and convex, this implies $Du(y)\in H^{-1}([0,k]).$ 
\end{proof}
\subsection{Additional Remarks}\label{rem2} The terminology \emph{absolutely subminimizing}, etc, is used here for the first time.  However, the notion is implicit in \cite[Proposition 4.4]{tour}. The cone functions of Section \ref{cb} are the same as the ``generalized cones" of \cite{GWY}, which, in turn, were the natural objects to consider when generalizing the cone comparison results of \cite{CEG}.  Definition \ref{dcc} differs in details from that used in \cite{GWY}, but they turn out to define the same notion.  Lemma \ref{htoc} is contained in Propositions 2.9 and 2.10 of Champion and De Pascali \cite{CD}, but this is not evident at a glance owing to layers of definitions, and the latter represent a long road to the simple result we need. The analogous point in \cite{GWY} is made with viscosity solution techniques.  On page 278 of \cite{C} one finds  a more portable tool.  It is shown that if $u\in\lloc(U)$ and $z:[0,1]\ra U$ is Lipschitz continuous and $N$ is a Lebesgue null set of $U$ containing the set of points at which $u$ is not differentiable, then there exists a function $g:[0,1]\ra \R^n$ such that 
\beqs
\frac{d\ }{dt} u(z(t))=g(t)\cdot \frac{dz}{dt}(t)
\eeqs
and 
\beqs
g(t)\in \bigcap_{r>0} \is{closed convex hull of} Du(B(z(t),r)\setminus N) 
\eeqs
for almost every $t\in[0,1]. $ By choosing $N$ to include the points where $H(Du)\le k$ fails, we see that \EQ{optconest} follows from \EQ{esth}.

\section{Absolutely Subminimizing to the Convexity Criterion via Cones}\label{amtctc}

The main results of this section establish that if $u \in C(\uc)$ is absolutely subminimizing, then $u$ satisfies the convexity criterion.  An intermediate role is played by the notion of comparisons with cones, and we first establish that absolute subminimizers have this property. Given Lemma \ref{htoc}, the proof parallels that of \cite{tour} in the case of norms.  

\begin{prop}\label{aml-cca}
If $u\in\lloc(U)$ is absolutely subminimizing in $U,$ then $u$
satisfies comparisons with cones from above in $U.$ \end{prop}
\begin{proof}  We argue by contradiction.
Suppose that $u \in \lloc(U)$ is absolutely subminimizing, but that it
does not satisfy comparisons with cones from above in $U.$ Then there
exists $V \ll U$ and a cone function $\varphi(x): = C_k(x-x_0) $ with
$k\geq 0$ and $x_0 \in \R^n \setminus V$ for which the comparison with
cones property fails; altering $u$ by a constant, we may assume that
\begin{equation*}
u(x) \leq \varphi(x) \quad \mbox{for every} \ x\in \partial V,
\end{equation*}
but    $u(x^*) > \varphi(x^*)$ for some $x^*\in V.$   We may assume as
well that $V$ is connected, $u> \varphi$ in $V$ and $u = \varphi$ on
$\partial V,$ since otherwise we may replace $V$ with the component of
$V \cap \{ u > \varphi\}$ containing $x^*.$

Consider the ray $R := \{ x_0 + t(x^*-x_0): t \geq 0 \}.$ Let $x_1,x_2
\in \partial V $ be  such that the connected component of $R \cap V$
containing $x^*$ is the line segment $( x_1, x_2 ),$ and $x_1$ is
contained in the line segment $[x_0,x_2).$ It follows that $x_1,
x_2\in\bv,$ and then
\begin{equation} \label{amlcca}
\begin{aligned}
C_k(x^*-x_1) & = C_k(x^*-x_0) - C_k(x_1-x_0) \\
                & = \varphi(x^*) - \varphi(x_1) < u(x^*) - u(x_1).
\end{aligned}
\end{equation}
Define $l:= \| H(Du) \|_{L^\infty(V)}$. We claim that $l > k$. Once
this is verified, and recalling \EQ{hdck}, we will obtain a
contradiction to our assumption that $u$ is absolutely subminimizing,
completing the proof. It suffices to show that
\begin{equation} \label{amlcccca}
u(x^*) - u(x_1) \leq  C_l(x^*-x_1),
\end{equation}
which will violate \EQ{amlcca} unless $l > k$. Fix $y \in (x^* ,
x_1),$ and notice that by Lemma \ref{htoc} we have
\begin{equation*}
u(x^*) - u(y) \leq C_l(x-y).
\end{equation*}
By sending $y \to x_1,$ we obtain \EQ{amlcccca}.
\end{proof}

The next result then takes us from absolutely subminimizing to the convexity criterion. However, the proof is rather lengthy, involving a tower of preliminary results which follow the statement.  This result was proven by Juutinen and Saksman \cite{JS} in the case that $H \in C^2(\R^n \setminus \{ 0\}) \cap C^1(\R^n)$, $H$ is convex in $\R^n$ and locally uniformly convex in $\R^n \setminus \{ 0\}$, and $H$ grows superlinearly in $p$.

\begin{prop}
\label{cca-conv}
Assume that $u \in C(U)$ is bounded and satisfies comparisons with cones from above in $U,$ and $r>0.$ Then there is an $\eta>0$ such that for  $x\in U_r$ 
\begin{equation}\label{conva}
\mbox{the map} \quad t \mapsto \flow^tu(x) \quad \mbox{is convex on} \quad [0,\eta)
\end{equation}
where $\eta = \eta(\osc_U u, r) > 0$ depends only on $\osc_U u$ and $r.$ 
\end{prop}

For each $u \in C(U)$ and $x\in U,$ we define the quantity
\begin{equation}  \label{Ssupfinite}
S^+ u(x) : = \limsup_{t \downarrow 0} \frac{\flow^t u(x) - u(x)}{t}.
\end{equation}  
It is involved in the formulation of the next lemma, and  is of further use later. Proposition \ref{cca-conv} is a direct consequence of Lemma \ref{HJconvex} below, as we demonstrate immediately below.  Then we will continue with a sequence of lemmata and culminate with the proof of Lemma \ref{HJconvex}.  

\begin{lem} \label{HJconvex}
Assume that $u \in C(U)$ is bounded and \ccas\ in $U.$  Then for each $r>0$  there exists $\eta > 0,$ depending only on $\osc_U u,$ $r,$ and the function $M$ of \EQ{pm}, such that
\begin{equation}\label{sple}
S^+u(x) \leq \frac{\flow^t u(x) - u(x)}{t} \quad \mbox{for all} \ x \in U_{r}, \ 0 < t < \eta.
\end{equation}
\end{lem}

\begin{proof}[{\bf Proof of Proposition \ref{cca-conv}}] Let $x\in U$ and $0<s<t<t_0(\osc_Uu,\dist(x,\bu))$  where $t_0$ is from Lemma \ref{local}.
The idea of the proof is to use \EQ{sple} with $T^su$ in place of $u$ and $t-s$ in place of $t$ to conclude that  
\begin{equation}
\label{ARGHHH}
\begin{split}
\limsup_{h \downarrow 0} \frac{\flow^h(\flow^su )(x) - \flow^su (x)}{h} &\leq \frac{\flow^{t-s}(\flow^su )(x) - \flow^s u(x)}{t-s}\\&=\frac{\flow^tu(x)-\flow^s(x)}{t-s}. 
\end{split}
\end{equation}
Expressed in terms of the function $\varphi(t) := \flow^tu(x),$ the inequality \EQ{ARGHHH} states that
\begin{equation}
\label{arggghhh}
\limsup_{h \downarrow 0} \frac{\varphi(s+h) - \varphi(s)}{h} \leq \frac{\varphi(t) -\varphi(s)}{t-s} \quad \mbox{for all} \ 0 < s < t.
\end{equation}
The convexity of $\varphi$ is guaranteed by \EQ{arggghhh} if $\varphi$ is Lipschitz. Indeed, the Lipschitz continuity of $\varphi$ and \EQ{arggghhh} imply that 
\beqs
\frac{d\ }{ds}\left[s\mapsto\frac {\varphi(t)-\varphi(s)}{t-s}\right]=\frac1{t-s}\lp -\varphi'(s)+\frac {\varphi(t)-\varphi(s)}{t-s}\rp\ge 0
\eeqs
almost everywhere. However, $\varphi$ is convex on an interval $[0,\delta)$ precisely when the map $s\mapsto (\varphi(t)-\varphi(s))/(t-s)$ is nondecreasing on $(0,t)$ for each $t < \delta;$ that is, the slope of a secant line is nondecreasing as a function of the left endpoint. 

It remains to do enough bookkeeping to justify the formal calculations above.  
Suppose that $x\in U_{2r}$ and $0<s<t_0(\osc_Uu,r).$    
First notice that $\flow^s u$ satisfies comparisons with cones from above in $U_{2r}.$    Indeed, for such $s$ and $x\in U_{2r}$ we may write
\begin{equation*}
\flow^s u(x) = \sup_{z \in B(0,r)} \left( u(z+x) - sL\left( \frac{z}{s} \right) \right).
\end{equation*} 
For  $z\in B(0,r),$ the map $x\mapsto u(z+x)$ has $A_r$ as a Lipschitz constant on $U_{2r},$ where $A_r$ is from Remark \ref{onur}; thus so does $\flow^su.$ Similarly, $\flow ^s u$ \ccas\ on $U_{2r},$ being the Lipschitz continuous supremum of functions with this property.  Finally, $\osc_U\flow^s u\le \osc_U u$ by \EQ{flowup}.    Thus  if $x\in U_{3r}$ we may apply Lemma \ref{HJconvex} with $U_{2r}$ in place of $U$ and $\flow^s u$ in place of $u$ to obtain \EQ{ARGHHH} if $t<\min(t_0(\osc_Uu,r), \eta).$ Here we are implicitly using Remark \ref{indu}. 

It remains to show that $\varphi(t)$ is Lipschitz.  If $0<s<t<t_0(\osc_Uu,r),$ and $x\in U_{r},$ then 
\beqs
\flow^s u(x)\le \flow^{t-s}\flow^su(x)=\flow^t u(x).
\eeqs
Thus $\varphi(t)$ is nondecreasing.  In the other direction, we use Remark \ref{stlip} to conclude that if $x\in U_{3r},$ then
\beqs
\flow^tu(x)-\flow^su(x)=\flow^{t-s}\flow^su(x)-\flow^s u(x)\le aK(t-s)
\eeqs
where $K=A_r$ is the Lipschitz constant for $\flow^su$ noted above and $a$ is from the remark.  This establishes the Lipschitz continuity of $\varphi$, and we have proved \EQ{conva} for $x\in U_{3r}.$ However, $r>0$ is at our disposal, and we are done. 
 \end{proof}

To prove Lemma \ref{HJconvex}, we require information about the sets 
\begin{equation*}
\Gamma_k := \{ q \in \R^n : q \in \partial H(p) \ \mbox{for some}  \ p \in H^{-1}(k) \},
\end{equation*}
and
\begin{equation*}
W_k := \{ q \in \R^n : q \in \partial H(p) \ \mbox{for some} \ p \in H^{-1}([0,k]) \},
\end{equation*}
where $k\geq 0.$ Observe that both $\Gamma_k$ and $W_k$ are closed and bounded.

\begin{lem}\label{neighbor}
For each $k > 0,$ the set $N_k := W_k \setminus \Gamma_k$ is a bounded neighborhood of the origin, and $\partial N_k \subseteq \Gamma_k.$
\end{lem}

\begin{proof}
Suppose $q_0 \in N_k$ and choose $p_0 \in H^{-1}[0,k)$ such that $q_0 \in \partial H(p_0).$ The affine function
\begin{equation*}
A_0(p) := H(p_0) + q_0 \cdot (p - p_0)
\end{equation*}
satisfies $A_0 \leq H,$ and moreover $q_0 \in \partial H(p)$ for every $p$ such that $A_0(p) = H(p).$ Since $q_0 \not\in \Gamma_k,$ there exists $\delta > 0$ such that $A_0(p) \leq H(p) - \delta$ for every $p\in H^{-1}(k).$

Let $R = \sup_{ p \in H^{-1}(k) } |p - p_0|$ and suppose $q \in B(q_0,\delta / 2R).$ Define
\begin{equation*}
A_1(p) := H(p_0) + q \cdot (p - p_0),
\end{equation*}
and observe that for any $p \in H^{-1}(k)$ we have
\begin{equation*}
H(p) - A_1(p) = H(p) - A_0(p) + (q - q_0) \cdot (p - p_0) \geq \delta / 2 > 0.
\end{equation*}
Thus there is a $p_1 \in H^{-1}[0,k)$ such that $H(p_1) - A_1(p_1) = \min_{ p\in H^{-1}[0,k] } (H - A_1)(p).$ Since $q \in \partial H(p_1),$ it follows that $q \in W_k.$ Therefore $B(q_0,\delta/2R) \subseteq W_k.$

We have shown that $N_k \subseteq \intr(W_k).$ Since $\Gamma_k$ is closed, it follows that $N_k$ is open. Since $W_k$ is closed, we deduce that $\partial N_k \subseteq \Gamma_k.$
\end{proof}

The following lemma connects the convexity criterion to comparisons with cones, and together with Lemma \ref{neighbor} plays a role similar to that of Proposition 2.5 in \cite{JS}.

\begin{lem} \label{coneflowl}
For every $k \geq 0,$ $x \in U,$ and $t> 0$ sufficiently small,
\begin{equation} \label{coneflow}
\flow^t C_k(x) = C_k(x) + kt.
\end{equation}
Moreover, for every $y\in \R^n$ we have
\begin{equation} \label{coneL}
C_k(y) \leq kt + tL\left(\frac{y}{t} \right),
\end{equation}
and equality holds in \EQ{coneL} provided that $y \in t\Gamma_k.$
\end{lem}
\begin{proof}
Observe that for any $y \in \R^n,$ 
\begin{align*}
C_k(y) - t L\left( \frac{y-x}{t} \right) & = \max_{p \in H^{-1}(k)} (p\cdot y) - \sup_{p\in \R^n} \left( p\cdot(y-x) - tH(p) \right) \\
& \leq \max_{p \in H^{-1}(k)} (p\cdot y) - \max_{p\in H^{-1}(k)} \left( p\cdot(y-x) \right) + kt \\
& \leq - \min_{p\in H^{-1}(k)} (-p\cdot x) + kt \\
& = C_k(x) + kt.
\end{align*}
Thus $\flow^t C_k(x) \leq C_k(x) + kt$ for all $t> 0$ and with respect to any domain $U.$ Taking $x=0$ in the above calculation, we obtain \EQ{coneL}.

Next, select $p_0 \in H^{-1}(k)$ such that $p_0\cdot x = C_k(x)$. Let $q_0\in \partial H(p_0)$ and put $y=x+tq_0,$ where $t>0$ is sufficiently small so that $y\in U;$  
then  $y-x \in 
t\partial H(p_0)$. It follows that
\beqs 
H(p) \geq H(p_0) + \frac{y-x}{t}\cdot (p - p_0) = k + \frac{y-x}{t}\cdot (p - 
p_0)\is{for all}p \in \R^n, 
\eeqs 
and hence 
\begin{equation*} 
L\left( \frac{y-x}{t} \right) = p_0 \cdot \frac{y-x}{t} - k. 
\end{equation*} 
Therefore, 
\begin{equation} \label{flosd} 
\begin{aligned} 
\flow^tC_k(x) & \geq C_k(y) - t L\left( \frac{y-x}{t} \right) \\ 
& = \sup_{p \in H^{-1}(k)} \left( (p-p_0)\cdot (y-x) + p\cdot x\right) + kt \\ 
& \geq p_0 \cdot x + kt\\ 
& = C_k(x) + kt. 
\end{aligned} 
\end{equation} 
This verifies \EQ{coneflow}, and implies that we must have equality in 
every line of \EQ{flosd}. Setting $x=0,$ we deduce that equality in 
\EQ{coneL} holds provided that $y\in t\partial H(p_0)$. However, in 
the case $x=0,$ we may take $p_0$ arbitrarily in $H^{-1}(k),$ and 
therefore we see that equality holds in \EQ{coneL} for any $y\in t 
\Gamma_k$.
\end{proof}

\begin{proof}[{\bf Proof of Lemma \ref{HJconvex}}]
For convenience in writing, we will establish \EQ{sple} for $x\in U_{2r}$ rather than $x\in U_r;$ as $r>0$ is arbitrary, there is no difference.   By Remark \ref{onur}, $u$ is Lipschitz continuous on $U_{r},$ with constant $K=A_{r}.$  Let $\eta > 0$ be so small that $ \eta \leq t_0(\osc_U u, r)$ and $\eta   N_{aK} \ll  B(0,r),$ where $a$ is the number from Remark \ref{stlip} and $N_{aK}$ is the neighborhood of the origin from Lemma \ref{neighbor}. Fix $0 < t < \eta,$ $x\in U_{2r},$ and set
\begin{equation*}
k:= \frac{\flow^t u(x) - u(x)}{t}.
\end{equation*}
Then $k \leq aK,$ and
\begin{equation}\label{magic}
\sup_{y\in B(x,r)} \left( u(y) - u(x) - t L\left( \frac{y-x}{t} \right) \right) = T^t u(x) - u(x) = kt.
\end{equation}
According to Lemma \ref{coneflowl},
\begin{equation*}
C_k(y-x) = kt + tL\left( \frac{y-x}{t} \right) \quad \mbox{for every} \ y\in x+ t \Gamma_k \supseteq x + t \partial N_k,
\end{equation*}
and since $x + t N_k \subseteq x + \eta  N_{aK} \ll B(x,r),$ we deduce from this identity in conjuction with \EQ{magic} that
\begin{equation*}
u(y) - u(x) \leq C_k (y-x) \quad \mbox{for every} \ y\in x + t\partial N_k.
\end{equation*}
Since $u\in \CCA(U),$ we therefore have
\begin{equation} \label{ccatoconvq}
u(y) - u(x) \leq C_k (y-x) \quad \mbox{for every}  \ y \in x + t N_k.
\end{equation}
Suppose that $0 < s < t_0(\osc_U u, \tilde{r}),$ where $\tilde{r}$ is so small that $B(x,\tilde{r}) \subseteq x + tN_k.$ Then according to \EQ{coneL} and \EQ{ccatoconvq},
\begin{align*}
\flow^s u(x) - u(x) & = \sup_{y \in B(x,\tilde r)} \left( u(y) - u(x) - sL\left(\frac{y-x}{s} \right) \right) \\
& \leq \sup_{y \in B(x,\tilde{r})} \left( C_k(y-x) - sL\left(\frac{y-x}{s} \right) \right) \\
& \leq ks.
\end{align*}
Dividing by $s$ and sending $s \to 0,$ we obtain $S^+ u(x) \leq k,$ as desired.
\end{proof}

\section{The Convexity Criterion to Absolutely Subminimizing}\label{end}

In this section, we study the relationship between the flow $t\mapsto \flow^t u(x)$ and the quantity $\| H(Du) \|_{L^\infty(U)},$ and then explore some consequences of the convexity criterion. We introduce a weaker \emph{pointwise convexity criterion}, and show that functions satisfying this weaker convexity criterion are necessarily absolutely subminimizing (see Proposition \ref{pconv-aml} below). The section culminates in Theorem \ref{equivalences}, which asserts the equivalence of the notions of absolutely subminimizing, the convexity criterion, and comparisons with cones.

We begin by collecting a number of preliminary results needed for the proof of Proposition \ref{pconv-aml}, some of which are also needed in Section \ref{patch}.

\begin{lem} \label{annoying}
Suppose that $u\in \lloc(U)$ satisfies
\begin{equation*}
\osc_U u \leq \alpha \quad and \quad \| H(Du) \|_{L^\infty(U)} \leq k.
\end{equation*}
Then for every $x \in U_r,$ 
\begin{equation}\label{flowspeed}
\flow^tu(x) - u(x)   \leq kt \quad \mbox{for all} \ 0 < t < t_0(\alpha,r).
\end{equation}
\end{lem}

\begin{proof}
The main point is that
\begin{equation} \label{annoyclm}
u(y) - u(x) \leq kt + tL\left( \frac{y-x}{t} \right) \quad \mbox{provided} \ [x,y] \subseteq U, \  t>0.
\end{equation} 
Indeed, if $[x,y]\subset U,$ the first inequality below is from Lemma \ref{htoc},  the following equality then holds for some  $p\in H^{-1}(k)$ and the final inequality is via the definition of $L:$
\begin{align*}
u(y)-u(x)\le C_k(y-x)& =p\cdot(y-x)\\ & =kt+t\lp p\cdot\frac{y-x}t-H(p)\rp \le kt+t\la{y-x}{t}.
\end{align*}
The proof is completed by using \EQ{annoyclm} and Lemma \ref{local} to obtain
\begin{equation*}
\flow^tu(x) - u(x)  = \sup_{y \in B(x,r)} \left( u(y) - u(x) - t L \left( \frac{y-x}{t} \right)\right) \leq kt\end{equation*}
for every $x\in U_r$ and $0 < t < t_0(\alpha,r).$
\end{proof}

The following connection between the quantities $S^+ u$ and $H(Du)$ allows us to deduce absolutely minimizing properties from the convexity criterion. 

\begin{lem} \label{amlconv}
Suppose that $u\in \lloc(U)$ is bounded. Then 
\begin{equation} \label{amlconeq}
\sup_{x \in U} S^+u(x) = \| H(Du) \|_{L^\infty(U)}.
\end{equation}
\end{lem}

\begin{proof}
Suppose $u$ is differentiable at $x\in U.$ Select $p \in \R^n$ and observe that for sufficiently small $t>0$ we have
\begin{equation*}
\begin{aligned}
\frac{\flow^t u(x)- u(x)}{t} & = \sup_{y \in U}\left( \frac{u(y) - u(x)}{t} - L\left( \frac{y-x}{t}\right)\right) \\
& \geq \frac{u( x + tp) - u(x)}{t} - L (p).
\end{aligned}
\end{equation*}
By sending $t \to 0$, taking the supremum over all $p\in \R^n$ and then using \EQ{HbyL}, we deduce that
\begin{equation*}
S^+u(x) \geq \sup_{p \in \R^n} (Du(x)\cdot p - L(p)) = H(Du(x)).
\end{equation*}
Since $u$ is differentiable almost everywhere by Rademacher's theorem, we deduce that 
\begin{equation*}
\sup_{x \in U} S^+u(x) \geq \| H(Du) \|_{L^\infty(U)}.
\end{equation*}

The reverse inequality follows from the previous lemma. Indeed, if $x \in U,$ then by \EQ{flowspeed} we have
\begin{equation*}
\flow^tu(x) - u(x)  = \sup_{y \in B(x,r)} \left( u(y) - u(x) - t L \left( \frac{y-x}{t} \right)\right) \leq t \| H(Du) \|_{L^\infty(U)}
\end{equation*}
for all sufficiently small $t > 0.$ Dividing by $t$ and passing to the limit as $t\downarrow 0,$ we obtain $S^+u(x) \leq  \| H(Du) \|_{L^\infty(U)}$ for every $x \in U.$
\end{proof}

\begin{lem}\label{Slip}
Suppose that $u\in C(U)$ satisfies the convexity criterion. Then for every $x \in V \ll U,$ the map 
\begin{equation} \label{flowinc}
t \mapsto \frac{\flow^t u(x) - u(x)}{t} \quad \mbox{is nondecreasing on the interval} \ [0,\delta],
\end{equation}
where $\delta = \delta(V)> 0$ is as in Definition \ref{dccri}. In particular
\begin{equation}\label{flowinca}
S^+u(x) = \inf_{0<t<\delta(V)} \frac{T^tu(x) -u(x)}{t} \quad  \mbox{for every} \  x \in V,
\end{equation}
and  the map $x \mapsto S^+u(x)$ is upper semicontinuous in $U.$ Moreover, $u \in \lloc(U).$
\end{lem}

\begin{proof}
The condition \EQ{flowinc} follows  from the convexity criterion, and \EQ{flowinca} is then immediate.  The upper semicontinuity of $x\mapsto S^+u(x)$ follows from \EQ{flowinca} and Remark \ref{stinx}.  To demonstrate the local Lipschitz continuity of $u,$ select $V \ll U.$ By \EQ{flowinc} we have
\begin{equation} \label{kVbnd}
k:= \sup_{x\in V, 0 < t \le \delta} \frac{\flow^t u(x) - u(x)}{t} = \sup_{x\in V} \frac{\flow^{\delta} u(x) - u(x)}{\delta}< \infty,
\end{equation}
where $\delta = \delta(V) > 0.$ Select $r_1>0$ so small that $\bar B(0,r_1) \subseteq \delta N_k,$ where $N_k$ is the neighborhood of the origin from Lemma \ref{neighbor}. Select $x,y\in V$ such that $|x-y| \leq r_1.$ Then we may select $0 < t < \delta$ such that $y-x \in t \Gamma_k,$ with $\Gamma_k$ also as in Lemma \ref{neighbor}. According to \EQ{kVbnd} and Lemma \ref{coneflowl}, we find that
\begin{equation*}
u(y) - u(x) \leq tk + L\left(\frac{y-x}{t} \right) = C_k(y-x) \leq K_k |y-x|.
\end{equation*}
Reversing the roles of $x$ and $y,$ we deduce that
\begin{equation*}
|u(x) - u(y)| \leq K_k|x-y| \quad \mbox{for every} \ x,y\in V, \ |x-y| \leq r_1.
\end{equation*}
Since $u$ is continuous, we have $\osc_V u < \infty,$ and thus 
\begin{equation*}
|u(x) - u(y)| \leq \osc_V u \leq \frac{\osc_V u}{r_1} |x-y| \quad \mbox{for every} \ x,y\in V, \ |x-y| \geq r_1.
\end{equation*}
Therefore $u$ is Lipschitz on $V$ with constant $\max\!\left\{ K_k, \osc_V u / r_1 \right\}$. Since $V \ll U$ is arbitrary, we deduce that $u \in \lloc(U)$.
\end{proof}

We now introduce a \emph{pointwise} version of the convexity criterion, which will turn out to be equivalent to the usual convexity criterion. It is convenient to use this pointwise notion when verifying that a given function is absolutely subminimizing, as it is apparently weaker. It is used for example in the proof of the patching lemma in the next section.

\begin{definition}\label{ptwisecc}
We say that a bounded function $u\in\lloc(U) $ \emph{satisfies the pointwise convexity criterion} in $U$ provided that the map $x \mapsto S^+u(x) $ is upper semicontinuous in $U,$ and that for each $x\in U$ there exists $\delta = \delta(x)>0$ such that 
\beq{Pccri}
\mbox{the map} \ t\mapsto \flow^tu(x)\quad \mbox{is convex on the interval}  \  [0,\delta(x)].
\eeq
\end{definition}

\medskip

According to Lemma \ref{Slip}, the convexity criterion is stronger than the pointwise convexity criterion. Below in Proposition \ref{pconv-aml}, we see that the pointwise convexity criterion is sufficient for the absolutely subminimizing property, and thus equivalent to the convexity criterion. With this end in mind, we cannot weaken Definition \ref{ptwisecc} by removing the hypothesis that $x\mapsto S^+u(x)$ is upper semicontinuous, as the following simple example demonstrates.

\begin{example} \label{notpccri}
Consider $H_1(p) = |p|$ and the function $u(x) = -|x|.$ It is easy to verify that with respect to $U=B(0,1),$
\begin{equation*}
\flow^t u(x) = \sup_{y \in B(x,t)} u(y) = \left\{ \begin{array}{ll}
t - |x| & \mrbox{for} x \neq 0\mrbox{and} 0\leq t<|x|, \\
0 & \mrbox{for} x = 0 \mrbox{and} 0\leq t<1.
\end{array} \right.
\end{equation*}
Thus for every $x\in U,$ the map $t\mapsto \flow^t u(x)$ is linear (and hence convex) on an  interval $[0,\delta(x)],$ but it is easy to see that $u$ is not absolutely subminimizing for $H_1$ in $B(0,1).$ This is due to the failure of the map $x\mapsto S^+u(x)$ to be upper semicontinuous, since 
\begin{equation*}
S^+u(x) = \begin{cases}
1 & \mbox{if} \ x \in B(0,1) \setminus \{ 0 \}, \\
0 & \mbox{if} \ x = 0.
\end{cases}
\end{equation*}
\end{example}

The next result, in other guises,  is a well-known and important technical tool  in the theory of absolutely minimizing functions born  in \cite{CEG}. Here we state a version in terms of the convexity criterion. More common in the literature is a slightly different result put in terms of cones; see for example \cite[Proposition 3.4]{GWY}.

\begin{lem}[Increasing slope estimate]\label{Pflowslope} 
Assume that $u\in\lloc(U)$ is bounded and satisfies the pointwise convexity criterion \EQ{Pccri}. Suppose that $x,y \in U$ and $0<t<\delta(x)$ are such that 
\begin{equation} \label{ysppt}
\flow^t u(x) = u(y) - t L\left( \frac{y-x}{t} \right).
\end{equation}
Then 
\begin{equation} \label{incsl}
\frac{\flow^tu(x) - u(x)}{t} \leq  S^+u(y).
\end{equation}
\end{lem}
\begin{proof}

Set $z_\lambda:= \lambda x + (1-\lambda) y$ for every $0 < \lambda < 1.$ Observe that
\begin{align*}
\flow^{(1-\lambda)t} u(x) & \geq u(z_\lambda) - (1-\lambda)t L\left( \frac{z_\lambda-x}{(1-\lambda)t} \right) \\
& = u(z_\lambda) - (1-\lambda) t L\left( \frac{y-x}{t} \right) \\
& = u(z_\lambda) - (1-\lambda) u(y) + (1-\lambda) \flow^t u(x). \\
& \geq u(z_\lambda) - (1-\lambda) u(y) + \flow^{(1-\lambda)t} u(x) - \lambda u(x),
\end{align*}
where the last inequality is obtained from the pointwise convexity criterion. By rearranging this inequality, we obtain
\begin{equation} \label{blel}
u(z_\lambda) \leq \lambda u(x) + (1-\lambda) u(y).
\end{equation}
Using \EQ{blel}, we see that
\begin{align*}
\flow^{\lambda t} u(z_\lambda) - u(z_\lambda) & \geq u(y) - \lambda t L\left( \frac{y-z_\lambda}{\lambda t} \right) -\lambda u(x) - (1-\lambda) u(y) \\
& = \lambda \left( u(y) - t L\left( \frac{y-x}{t} \right) -  u(x) \right) \\
& = \lambda \left( \flow^t u(x) - u(x) \right).
\end{align*}
Dviding by $\lambda t,$ we obtain
\begin{equation} \label{bleg}
\frac{\flow^t u(x) - u(x)}{t} \leq  \frac{\flow^{\lambda t} u(z_\lambda) - u(z_\lambda)}{\lambda t} \quad \mbox{for every} \ 0 < \lambda < 1.
\end{equation}
Select $0 < r < \dist(y, \partial U).$ For sufficiently small $\lambda > 0,$ we have $z_\lambda \in B(y, r/2)$ and $\lambda t < t_0\left(\osc_U u,r/2\right),$ and for such $\lambda$ we have
\begin{equation}\label{fgest}
\frac{\flow^{\lambda t} u(z_\lambda) - u(z_\lambda)}{\lambda t} \leq  \| H(Du) \|_{L^\infty(B(y, r))} = \sup_{z \in B(y, r)} S^+u(z)
\end{equation}
by Lemmata \ref{annoying} and \ref{amlconv}. Combining \EQ{bleg} and \EQ{fgest} yields
\begin{equation} \label{blek}
\frac{\flow^t u(x) - u(x)}{t} \leq \sup_{z \in B(y, r)} S^+u(z).
\end{equation}
Recall that the map $x\mapsto S^+u(x)$ is upper semicontinuous since $u$ satisfies the pointwise convexity criterion. Thus by sending $r \downarrow 0$ in \EQ{blek} we obtain \EQ{ysppt}.
\end{proof}

Our argument for the following proposition is similar in spirit to the one found in Section 4 of \cite{JS}.

\begin{prop}\label{pconv-aml}
Suppose that $u\in\lloc(U)$ is bounded and satisfies the pointwise convexity criterion \EQ{Pccri}. Then $u$ is absolutely subminimizing in $U.$
\end{prop}
\begin{proof}
Assuming that $u$ is not absolutely subminimizing, we select $V \ll U$ and $v \in \lloc(V)$ such that $u \geq v$ in $V,$ $u=v$ on $\partial V,$ and
\begin{equation*}
k: = \| H(Dv) \|_{L^\infty(V)} < \| H(Du) \|_{L^\infty(V)}.
\end{equation*}
According to \EQ{amlconeq}, we may rewrite this as
\begin{equation*}
k = \sup_{x\in V} S^+ v(x) < \sup_{x\in V} S^+ u(x).
\end{equation*}
Take $l$ such that $k < l < \sup_{x\in V} S^+u(x)$ and define $E: = \{ x \in \bar V : S^+u(x) \geq l \}.$ Notice that $E$ is closed since $x\mapsto S^+u(x)$ is upper semicontinuous, and $E\cap V$ is nonempty. We claim that there exists $x\in E\cap V$ such that 
\beq{achieved}
u(x)-v(x)=m:=\max_E (u-v). 
\eeq
If $m=0$ we may take any point $x\in E\cap V,$ since $u-v \geq 0$ in $V.$ If $m>0,$ we select any $x\in E$ satisfying \EQ{achieved}, since in this case $x\not\in \partial V$ due to the fact that $u-v = 0$ on $\partial V.$

We now proceed to derive a contradiction from \EQ{achieved}. By the pointwise convexity criterion there is $\delta(x) > 0$ such that $ t\mapsto \flow^tu(x)$ is convex on the interval $[0, \delta(x)].$ Select a small $0 < t < \max\{ \delta(x), t_0(\osc_U u, \dist(x,\partial V))\}$ such that 
\begin{equation} \label{flowvu}
\frac{\flow^t v(x) - v(x)}{t} < l,
\end{equation}
and choose $y \in B(x,\dist(x,\partial V) )$ so that
\begin{equation}\label{yptinc}
\flow^t u(x) = u(y) - t L\left( \frac{y-x}{t} \right).
\end{equation}
The monotonicity property \EQ{flowinc} is clearly valid on the interval $[0,\delta(x))$ for a function $u$ satisfying the pointwise convexity criterion. Using this together with Lemma \ref{Pflowslope}, we obtain
\beqs
S^+u(y) \geq \frac{\flow^t u(x) - u(x)}{t} \geq S^+u(x) \geq l.
\eeqs
Thus $y \in E.$ Using \EQ{flowvu} and \EQ{yptinc}, we see that
\begin{align*}
v(y) & \leq \flow^t v(x) + t L\left( \frac{y-x}{t} \right) \\
& < t l + v(x) + t L\left( \frac{y-x}{t} \right) \\
& \leq t S^+u(x) + u(x) + t L\left( \frac{y-x}{t} \right) - (u(x) - v(x))\\
& \leq \left( \flow^t u(x) - u(x) \right) + u(x) + t L\left( \frac{y-x}{t} \right)  - (u(x) - v(x)) \\
& = u(y)  - (u(x) - v(x)) .
\end{align*}
Thus $u(y) - v(y) >  u(x) - v(x)=m,$ contradicting the definition of $m.$ 
\end{proof}

At this point, we have proved the equivalence of the notions of absolutely subminimizing, comparisons with cones from above, the convexity criterion, and the pointwise convexity criterion. We summarize this below in Theorem \ref{equivalences}. The result is new insofar as we  make no regularity assumptions on $H,$ split the definition of absolute minimizer into two halves, and include a \emph{pointwise} convexity criterion.

The equivalence between absolute minimizers and functions satisfying comparisons with cones was first proved for $H_2(p) = |p|^2$ in \cite{CEG}, and for a more general $C^2$ Hamiltonian in \cite{GWY}. The equivalence between the convexity criterion and viscosity subsolutions of Aronsson's equation was proved by Juutinen and Saksman \cite{JS} for $H\in C^2(\R^n \setminus \{ 0 \}) \cap C^1(\R^n),$ which links the convexity and concavity criterion with absolute minimizers for $H\in C^2$ after taking into account the results in \cite{GWY}.

\begin{thm}\label{equivalences} Assume that $u:U\to \R$ is bounded. Then the following statements are equivalent:
\begin{enumerate}
\item $u$ is absolutely subminimizing in $U$;
\item  $u$ satisfies comparisons with cones from above in $U$;
\item $u$ satisfies the convexity criterion in $U$;
\item $u$ satisfies the pointwise convexity criterion in $U.$
\end{enumerate}
\end{thm}
\begin{proof}
Proposition \ref{aml-cca} is the assertion (i)$\implies$(ii), and Proposition \ref{cca-conv} states that (ii)$\implies$(iii). That (iii)$\implies$(iv) is a consequence of Lemma \ref{Slip}. Finally, Proposition \ref{pconv-aml} asserts that (iv)$\implies$(i).
\end{proof}

Conspicuously absent from the list of equivalences in Theorem \ref{equivalences} is a statement about $u$ being a viscosity subsolution of Aronsson's equation \EQ{ae}. We prove in Appendix \ref{viscsol} that conditions (i)-(iv) are sufficient for $u$ to be a viscosity subsolution of \EQ{ae}, while the necessity of the convexity criterion for the Aronsson equation is an open problem, as mentioned in the introduction.


\section{The patching lemma}\label{patch}

In this section we surmount the final technical hurdle and finish the proof of Theorem \ref{compth}. The following lemma states that we may approximate any function satisfying the pointwise convexity criterion with another such function $v$ which has the additional property that $S^+v > 0.$

\begin{lem}[Patching Lemma]\label{patching}
Suppose that $U$ is bounded and $u \in \lloc(U) \cap C(\uc)$ satisfies the pointwise convexity criterion. Then there is a family of functions $\{ u_\gamma \}_{\gamma>0} \subseteq \lloc(U)\cap C(\uc)$ with the following properties:
\begin{enumerate}
\item $u_\gamma=u$ on $\bu$ and $u_\gamma\le u$ on $\uc;$
\item $u_\gamma \to u$ uniformly on $\bar U$ as $\gamma\downarrow 0;$
\item for each $\gamma >0,$ the function $u_\gamma$ satisfies pointwise convexity criterion;
\item $S^+ u_\gamma(x) \geq \gamma$ for every $x\in U$ and $\gamma > 0$.
\end{enumerate}
\end{lem}

To prove the patching lemma, we require the following preliminary result. It is distinguished in that it is the only place in this article where we need to use the hypothesis that $H^{-1}(0)$ has empty interior.

\begin{lem}\label{prepatch}
Suppose that $U$ is bounded and $\ep > 0$ is fixed. There exists $k > 0,$ depending only on $H,$ $\diam(U),$ and $\ep,$ such that whenever $u,v\in C(\bar U)$ satisfy $u=v$ on $\partial U$ and
\begin{equation}\label{small}
\sup_{x\in U} \left( S^+ u(x) + S^+ v(x) \right) \leq k,
\end{equation}
then
\begin{equation*}
\max_{\bar U} |u - v| \leq \ep.
\end{equation*}
\end{lem}
\begin{proof}
Since the level set $H^{-1}(0)$ is convex and has empty interior, we may select a unit vector $q$ such that
\begin{equation*}
p\cdot q = 0 \quad \mbox{for every} \ p \in H^{-1}(0).
\end{equation*}
Since $0 \in H^{-1}(0),$ we may select a small $k > 0$ such that
\begin{equation} \label{qsmush}
C_k(\pm q) \leq \frac{\ep}{2 \diam(U)}.
\end{equation}
Given $x \in U,$ let $r_0 := \inf \{ r > 0 : x + r q \notin U \}$ and then $y := x + r_0 q \in \partial U.$ Since $r_0 \leq \diam(U),$ \EQ{small} and \EQ{qsmush} together with Lemmata \ref{amlconv} and \ref{htoc} imply
\begin{equation*}
|u(x) - u(y)| \leq \frac{\ep}{2} \quad \mbox{and} \quad |v(x) - v(y)| \leq \frac{\ep}{2}.
\end{equation*}
Since $u(y) = v(y),$ we see that $|u(x) - v(x)| \leq \ep.$
\end{proof}

\newcommand{\ug}{u_\gamma}
\begin{proof}[{\bf Proof of Lemma \ref{patching}}]
For each $\gamma > 0,$ consider the open set
\begin{equation*}
V_\gamma := \{x\in U:  S^+ u(x) < \gamma \}.
\end{equation*}
That $V_\gamma$ is open follows from the upper semicontinuity of $x\mapsto S^+u(x).$   For each $x\in\bar V_\gamma,$ let ${\mathcal P}(x)$  be the set of finite ordered lists 
\beqs
[x=x_0, x_1, \ldots, x_N] \is{such that} (x_{i},x_{i+1})\subset V_\gamma \is{for} i=0,\ldots N-1,\is{and} x_N\in\partial V_\gamma.
\eeqs
Here $N$ can take all nonnegative integer values,  $N=0, 1, 2,\ldots.$ 
Now we define a function $v_\gamma : \bar V_\gamma \to \R$ by
\begin{equation} \label{vgamdef}
v_\gamma(x) := \sup \left\{ u(x_N) - \sum_{i = 0}^{N-1} C_\gamma(x_{i+1}-x_i) : [x=x_0, \ldots, x_N]\in {\mathcal P}(x)\right\},
\end{equation}
and then define   $\ug : \bar U \to \R$ by
\begin{equation*}
\ug(x) = \begin{cases}
v_\gamma(x) & \mrbox{if} x \in \bar V_\gamma, \\
u(x) & \mrbox{if} x \in \bar U \setminus \bar V_\gamma.
\end{cases}
\end{equation*}
We show that the family $\{ \ug \}_{\gamma > 0}$ has the desired properties. 
\medskip

\noindent\emph{Claim 1: $v_\gamma \leq u$ in $V_\gamma$ and $v_\gamma = u$ on $\partial V_\gamma.$ In consequence, $\ug\le u$ in $\bar U$ and $\ug=u$ on $\bu.$ }

Suppose $[x_0, ..., x_N]\in \cp(x).$ Then Lemmata \ref{amlconv} and \ref{htoc} and $(x_i,x_{i+1})\subset V_\gamma$ and continuity of $u$ imply that
\begin{equation*}
u(x_N) - u(x_0) = \sum_{i=0}^{N-1} (u(x_{i+1}) - u(x_i)) \leq \sum_{i = 0}^{N-1} C_\gamma(x_{i+1} - x_i).
\end{equation*}
In consequence, $v_\gamma(x_0) \leq u(x_0).$ That $v_\gamma \geq u$ on $\partial V_\gamma$ is immediate: if $x\in\partial V_\gamma,$ then $[x]\in \cp(x).$

\medskip

\noindent \noindent\emph{Claim 2: $v_\gamma$ is continuous on $\bar V_\gamma.$ In consequence, $u_\gamma\in C(\uc).$ } First notice that if $(x,y)\subset V_\gamma,$  $\varepsilon>0,$ and $[x_0,\ldots, x_N]\in \cp(x)$ is such that
\beq{diseps}
v_\gamma(x)-\varepsilon\le u(x_N) - \sum_{i = 0}^{N-1} C_\gamma(x_{i+1}-x_i),
\eeq
then 
\beqs
[y_0,\ldots, y_{N+1}]:=[y,x_0,\ldots, x_N]\in \cp(y)
\eeqs
and 
\beqs \begin{split}
v_\gamma(y) & \ge u(y_{N+1})- \sum_{i = 0}^{N} C_\gamma(y_{i+1}-y_i) \\ & = u(x_N) -\sum_{i = 0}^{N-1} C_\gamma(x_{i+1}-x_i)-C_\gamma(x-y) \\ & \ge v_\gamma(x)-\varepsilon -C_\gamma(x-y). \end{split}
\eeqs
As $\varepsilon>0$ is at our disposal, we conclude that  if $(x,y)\subset V_\gamma,$ then 
\beq{vkc}
v_\gamma(x)-v_\gamma(y)\le C_\gamma(x-y).
\eeq
It follows that $v_\gamma$ is continuous on $V_\gamma.$ As $u$ is continuous on $\bar U,$ $v_\gamma=u$ on $\partial V_\gamma$ and $v_\gamma\le u$ on $\bar V_\gamma,$  continuity of $v_\gamma$ on $\bar V_\gamma$ is assured, provided that we show $v_\gamma$ is lower semicontinuous at points of $\partial V_\gamma$  when approached from inside $V_\gamma.$  Take $y\in\partial V_\gamma$ and $y_j\in V_\gamma$ such that $y_j\ra y$ as $j\ra \infty.$ Let $x_j\in \partial V_\gamma$ be a nearest point to $y_j.$ Then $(x_j,y_j)\subset V_\gamma$ and $x_j\ra y.$ According to \EQ{vkc} and $v_\gamma=u$ on $\partial V_\gamma,$ we have
\beqs
u(x_j)=v_\gamma(x_j)\le v_\gamma(y_j)+C_\gamma(x_j-y_j),
\eeqs
and therefore $\limsup_{j\ra\infty} v_\gamma(y_j)\ge u(y) = v(y).$

\medskip

\noindent \noindent\emph{Claim 3: For every $x \in V_\gamma$ and sufficiently small $t>0,$
\begin{equation} \label{linfl}
\flow^t\ug(x) -\ug(x) = \gamma t
\end{equation}
In particular, the map $t \mapsto T^t \ug(x)$ is convex on $[0, \delta]$ for small $\delta > 0$ and $S^+\ug(x) = \gamma.$}

Select $x \in V_\gamma$ and $0 < r < \dist(x, \partial V_\gamma).$ By Lemma \ref{local} there is $t_0 = t_0(\osc_U \ug, r)$ such that 
\begin{equation*}
\flow^t \ug(x) = \sup_{y\in  B(x,r)} \left( v_\gamma(y) - t L\left( \frac{y-x}{t} \right) \right)\!,
\end{equation*}
for every $0 <t < t_0.$ Using \EQ{vkc} and  Lemma \ref{coneflowl} we have that
\begin{equation*} 
\begin{aligned}
\flow^t \ug(x) & = \sup_{y\in  B(x,r)} \left( v_\gamma(y) - t L\left( \frac{y-x}{t} \right) \right) \\
& \leq \sup_{y\in  B(x,r)} \left( C_\gamma(y-x) + v_\gamma(x) - t L\left( \frac{y-x}{t} \right) \right) \\
& = v_\gamma(x) + \gamma t .
\end{aligned}
\end{equation*}

For the other direction, assume that $t >0$ is so small that $x + t \partial N_\gamma  \subset B(x,r)$ where $N_\gamma$ is as in Lemma \ref{neighbor},   fix a small $\ep >0,$ and choose $[x_0,\ldots, x_N]\in\cp(x)$ such that   \EQ{diseps} holds.  It is clear that there is a point $y \in  [x_j, x_{j+1}] \cap (x+ t\partial N_\gamma)$ for some $j \in\set{0,1, ..., N-1}.$ Then 
\beqs
[y,x_{j+1},\ldots,x_N]\in \cp(y),
\eeqs
so
\begin{equation*}
v_\gamma(y) \geq u(x_N) - \sum_{i = j+1}^{N-1} C_\gamma(x_{i+1} - x_i) - C_\gamma(x_{j+1}-y),
\end{equation*}
and then
\begin{equation*}
v_\gamma(y) - v_\gamma(x) + \ep \geq - C_\gamma(x_{j+1}-y) + \sum_{i = 0}^{j} C_\gamma(x_{i+1} - x_i).
\end{equation*}
Noting that $C_\gamma(x_{j+1}-x_j)-C_\gamma(x_{j+1}-y)=C_\gamma(y-x_j)$ by the choice of $y,$ we further have, by subadditivity, 
\begin{multline*}
- C_\gamma(x_{j+1}-y) + \sum_{i = 0}^{j} C_\gamma(x_{i+1} - x_i) \\ = C_\gamma(y-x_j) +C_\gamma(x_j-x_{j-1})+\cdots +C_\gamma(x_1-x) \ge C_\gamma(y-x).
\end{multline*}
Combining this with the previous inequality, we arrive at 
\beqs
v_\gamma(y)-v_\gamma(x)+\ep\ge C_\gamma(y-x). 
\eeqs
Hence 
\begin{equation*}
\begin{aligned}
\flow^t u_\gamma(x) - u_\gamma(x) + \ep & \geq  v_\gamma(y) - v_\gamma(x) -  t L\left( \frac{y-x}{t} \right) + \ep \\ 
& \geq C_\gamma(y - x) -  t L\left( \frac{y-x}{t} \right) \\
& = \gamma t
\end{aligned}
\end{equation*}
where the last equality follows from $y \in x + t \partial N_\gamma$ and Lemmata \ref{neighbor} and \ref{coneflowl}. The claim follows since $\ep >0$ was arbitrarily small. 

\medskip

\noindent \emph{Claim 4: For every $x \in U\setminus V_\gamma,$ the flow $t \mapsto \flow^t \ug(x)$ is convex on $[0,\delta]$ for sufficiently small $\delta=\delta(x)> 0.$ Moreover, $S^+\ug(x) =S^+u(x)$ on $U \setminus V_\gamma.$}

By the assumption there is $\delta(x)>0$ such that $t \mapsto \flow^t u(x)$ is convex on the interval $[0,\delta(x)].$ Select $r < \dist(x, \partial U)$ and $0 < t < \min \{t_0(\osc_U \ug, r), \delta(x) \}.$ Choose $y \in B(x,r)$ such that 
\begin{equation*}
\flow^t u(x) = u(y) - t L\left( \frac{y-x}{t} \right).
\end{equation*}
By the increasing slope estimate (\ref{incsl}) we have 
\begin{equation*}
S^+u(y) \geq \frac{\flow^tu(x) - u(x)}{t} \geq S^+u(x)
\end{equation*}
and therefore $y \in U\setminus V_\gamma.$ Since $\ug = u$ in $ U\setminus V_\gamma,$ we have that
\begin{equation*}
\flow^t u(x) = u(y) - t L\left( \frac{y-x}{t} \right) =  \ug(y) - t L\left( \frac{y-x}{t} \right) \leq  \flow^t \ug(x).
\end{equation*}
On the other hand, $u \geq \ug$ in $U,$ and therefore we deduce that for sufficiently small $t > 0,$
\begin{equation} \label{flowseq}
\flow^t \ug (x) = \flow^t u(x) \quad \mbox{for every} \ x \in U\setminus V_\gamma.
\end{equation}
The claims follow from this identity and the convexity of the map $t \mapsto \flow^t u(x).$ 

\medskip
\noindent \emph{Claim 5: The function $\ug$ satisfies the pointwise convexity criterion \EQ{Pccri}.}

For any $x\in U,$ the convexity of $t\mapsto \flow^t \ug(x)$ on a suitable interval is established in Claims 3 and 4. It remains to verify that the map $x \mapsto S^+\ug(x)$ is upper semicontinuous on $U.$ Since  $S^+\ug=S^+u$  on $U \setminus V_\gamma$ by Claim 4, and $x\mapsto S^+u(x)$ is upper semicontinuous by assumption, the restriction of $S^+\ug$ to $U\setminus V_\gamma$ is upper semicontinuous.   Since $S^+\ug=\gamma$ on $V_\gamma$ by Claim 3, $x\mapsto S^+\ug(x)$ is upper semicontinuous on $V_\gamma.$  It remains to argue that $x\mapsto S^+u_\gamma(x)$ is upper semicontinuous at points of $\partial V_\gamma.$ However, this follows from the fact that $S^+\ug=S^+u\ge \gamma$ on $\partial V_\gamma,$ by the definition of $V_\gamma,$ and $S^+\ug=\gamma$ on $V_\gamma.$ 
 \medskip

Having proven Claims 1-5, we complete the proof by observing that Lemma \ref{prepatch} ensures that $u_\gamma \to u$ uniformly on $\bar U$ as $\gamma \downarrow 0.$
\end{proof}

\begin{remark}\label{unbndp} The various steps in the proof of Lemma \ref{patching} are valid for unbounded $U,$ with the exception of the last sentence wherein Lemma \ref{prepatch} is invoked.  
This is because the proof of Lemma \ref{prepatch} fails if $U$ is unbounded. However, if $H^{-1}(0)=\set{0},$ then we may simply let $y$ in that argument be the point of $\bu$ nearest $x.$  Then, from Lemmata \ref{amlconv} and \ref{htoc} we have 
\beqs
u(x)-u(y)=u(x)-v(y)\le C_k(x-y) \iq{and} v(y)-v(x)\le C_k(y-x)
\eeqs
for $(x,y)\subseteq U,$ which implies that 
\beqs
|u(x)-v(x)|\le 2K_k|x-y|
\eeqs
where $K_k$ is from \EQ{ckl}. If $H^{-1}(0)=\set{0},$ then $K_k\ra 0$ as $k\downarrow0.$    In the context of Lemma \ref{patching}, this results in 
\beqs
\lim_{\gamma\downarrow 0} u_\gamma=u\iq{uniformly on compact subsets of} \uc. 
\eeqs
\end{remark}

\begin{remark}\label{unbndu} We have saved a delightful and unannounced surprise for the reader which we now serve up. Theorem \ref{compth} remains valid for unbounded $U,$ provided that $u$ and $v$ are bounded, the level set $H^{-1}(0)=\set{0}$, and $\bu$ is nonempty and compact. That is, we obtain a comparison result in \emph{exterior domains}. Typical examples are the exterior of a ball, $U=\set{x: |x|>R},$ or the complement in $\R^n$ of a finite number of points. The hypothesis that $H^{-1}(0) = \{ 0 \}$ is evidently necessary; see Remark \ref{globalc}.

Our approach corresponds to the use of patching in \cite{CGW} to establish uniqueness results for some cases in which $U$ is unbounded and $H$ is a norm. To establish this comparison result, we merely need to show that $u_\gamma(x)\ra-\infty$ as $|x|\ra \infty.$ Indeed, for in this case we have for all sufficiently large $R>0,$ 
\begin{equation*}
u_\gamma(x)-v(x)<\max_{\bu}(u-v) \iq{for} |x|\ge R,
\end{equation*}
and we have confined the maximum of $u_\gamma-v$ to a bounded set on which we may invoke Theorem \ref{compth}.  Then we send $\gamma\downarrow 0,$ and use the preceding remark. 

The first step is to note that if $\bu\subset B(0,R_0)$, $R>R_0$ and $|x| >R,$  then $B(x,R-R_0)\ll U,$ and by (the proof of) Lemma \ref{ccca-lip},
\beqs
u(y)-u(x)\le C_k(x-y) \iq{if} y\in B(x,R-R_0)\iq{and}  \frac{\osc_Uu}{R-R_0}\le M_k. 
\eeqs
Putting  $k=\gamma/2$ and choosing $R$ sufficiently large so that $\osc_Uu/(R-R_0)\le M_k$ and invoking (the proof of) Lemma \ref{htoc} together with Lemma \ref{amlconv}, we discover 
\beqs
S^+u(x)\le \frac\gamma 2\iq{if} |x|>R. 
\eeqs
Thus $\partial V_\gamma$ is bounded. We denote the value of $R$ involved here by $R_\gamma$ below, returning $R$ to other uses. The key observation is that $\partial V_\gamma\subset B(0,R_\gamma).$ 

With $\gamma$ fixed  and $R_\gamma$ as above, we examine the behavior of the function $v_\gamma$ defined in \EQ{vgamdef} outside of the ball $B(0,R_\gamma).$   Fix $\ep>0,$ let $|x|>R>R_\gamma,$ and let $[x_0, \ldots, x_N]\in \cp(x)$ have the property that
\beq{yep}
v_\gamma(x)-\ep\le u(x_N)-\sum_{i=0}^{N-1}C_\gamma(x_{i+1}-x_i).
\eeq
 Let $j\in \set{1,\ldots, N-1}$ be the smallest integer satisfying $[x_j,x_{j+1}]\cap \partial B(0,R_\gamma)\not=\emptyset$ and choose the point $y\in [x_j,x_{j+1}]\cap \partial B(0,R_\gamma)$ with the property that $[x_j,y)\cap \bar B(0,R_\gamma)=\emptyset.$ In particular, if $x_j\in \partial B(0,R_\gamma),$ then $y=x_j$.  Then $[y,x_{j+1},\ldots, x_N]\in P(y),$ and with \EQ{yep} this implies
 \beqs
 v_\gamma(x)-\ep\le u(x_N)-\sum_{i=j+1}^{N-1}C_\gamma(x_{i+1}-x_i)  -C_\gamma(x_{j+1}-x_j)- \sum_{i=0}^{j-1}C_\gamma(x_{i+1}-x_i).
 \eeqs
 Now we use that 
 \beqs
 C_\gamma(x_{j+1}-x_j)= C_\gamma(x_{j+1}-y)+C_\gamma(y-x_j)
 \eeqs
 and 
 \beqs
 u(x_N)-\sum_{i=j+1}^{N-1}C_\gamma(x_{i+1}-x_i)  -C_\gamma(x_{j+1}-y)\le v_\gamma(y)
 \eeqs
 in conjunction with the line above to conclude that
 \beqs
 \begin{split}
 v_\gamma(x)-\ep & \le v_\gamma(y)-C_\gamma(y-x_j) - \sum_{i=0}^{j-1}C_\gamma(x_{i+1}-x_i) \\ & \le v_\gamma(y)-C_\gamma(y-x) \le \max_{\partial B(0,R_\gamma)}v_\gamma \ -\min_{y\in \partial B(0,R_\gamma)}C_\gamma(y-x). 
 \end{split} \eeqs
 Using $M_\gamma$ from \EQ{ccoer} and that $\ep> 0$ is arbitrary, we can produce the more transparent estimate
 \beqs
 v_\gamma(x)\le A_\gamma- M_\gamma|x| \is{where} A_\gamma :=\max_{\partial B(0,R_\gamma)}v_\gamma+M_\gamma R_\gamma. 
 \eeqs
 In particular, $v_\gamma(x)\ra -\infty$ as $|x|\ra\infty.$ 
\end{remark}


\appendix

\section{The necessity of the Aronsson equation} \label{viscsol}

Aronsson \cite{A2} showed that the infinity Laplace equation
$\Delta_\infty u=0$ characterizes the absolute minimizers of 
$H(p)=|p|$ which are $C^2.$  Jensen \cite{RJ} perfected this
result by showing that if the infinity Laplace equation is understood in the viscosity
sense, then it completely characterizes the absolutely minimizing property.
The derivation of the Aronsson equation (in the viscosity sense) was
extended to the generality of absolutely minimizing functions for suitable
twice differentiable $H(x,s,p)$ in \cite{BJW1}  (see also \cite{C}), and
then for $H \in C^1$ in \cite{CWY}. In this appendix, in keeping with our
theme of emphasizing the convexity criteria, we derive the Aronsson
equation for functions satisfying the pointwise convexity criterion.
With Theorem \ref{equivalences} in mind, we deduce that absolute
subminimizers for $H$ satisfying \EQ{ch} are viscosity subsolutions of
the Aronsson equation in the sense defined in the introduction.

The converse problem is more difficult.  It is an open problem as to whether
viscosity solutions of the Aronsson equation for general convex $H(p)$ are
necessarily absolutely minimizing in general. This has been resolved in the affirmative  for twice differentiable $H(p)$ and $H(p,x)$ in \cite{GWY}  and Yu \cite{Yu2}, respectively. Examples of some nonsmooth $H$ for which the Aronsson equation is known to be sufficient for the absolutely minimizing property can be found in \cite{CGW}.

\medskip

Proceeding now with the proof, let us take a function $u\in \lloc(U)$
which satisfies the pointwise convexity criterion \EQ{Pccri}. To show that
$u$ is a viscosity subsolution of the
Aronsson equation \EQ{ae}, we fix a test function $\varphi \in
C^2(U)$ and a point $x_0\in U$ such that
\begin{equation*}
\mbox{the map} \quad x \mapsto (u-\varphi)(x) \quad \mbox{has a strict
local maximum at} \ x = x_0.
\end{equation*}
Our task is to demonstrate that
\begin{equation}\label{aenewts}
\omega \cdot D^2\varphi(x_0) \omega \geq 0 \quad \mbox{for some} \
\omega \in \partial H\!\left( D\varphi(x_0) \right).
\end{equation}
We may fix a small radius $r> 0$ for which $ B(x_0 ,r ) \ll U$ and
\begin{equation} \label{x0local}
(u-\varphi)(x_0) = \max_{x\in \bar B(x_0,r)} (u-\varphi)(x) >
\max_{x\in \partial B(x_0,r)} (u-\varphi)(x).
\end{equation}
First we notice that according to \cite[Proposition 2.3(i)]{CWY} followed by Lemma
\ref{amlconv} and the assumed upper semicontinuity of $x\mapsto S^+u(x),$ we have
\beq{intaene}
H(D\varphi(x_0))\le \lim_{s\downarrow 0} \esssup_{B(x_0,s)}H(Du)\le
S^+u(x_0).
\eeq

Take $\alpha > 0$ to be the greater of $\osc_{B(x_0,r)} u$ and
$\osc_{B(x_0,r)} \varphi$. According to Lemma \ref{local},
\EQ{flowinc}, \EQ{x0local}, and \EQ{intaene}, for every $0 < t < t_0(\alpha,r/2)$ we
have
\begin{equation}\label{aenedi}
\begin{aligned}
0 & \leq \frac{T^tu(x_0) - u(x_0)}{t} - S^+u(x_0) \\
& \leq \frac{T^t\varphi(x_0) - \varphi(x_0)}{t} -
H\!\left(D\varphi(x_0)\right) \\
& = \sup_{y\in B(x_0,r/2)} \left( \frac{\varphi(y)-\varphi(x_0)}{t} - L\left( \frac{y-x_0}{t} \right) - H\!\left(D\varphi(x_0) \right)
\right).
\end{aligned}
\end{equation}
For each $0 < t < t_0(\alpha,r/2),$ let $y_t$ be a point where the
maximum above is attained, and set $\omega_t := (y_t -x_0) / t$. From
\EQ{aenedi} we obtain
\begin{equation} \label{aeneditay}
0 \leq \omega_t \cdot D\varphi(x_0) + \frac{t}{2} \omega_t \cdot
\left( D^2\varphi(x_0) \omega_t\right) - L(\omega_t) - H\!\left(
D\varphi(x_0) \right) + o\left( t |\omega_t|^2 \right)
\end{equation}
as $t \downarrow 0$. According to Remark \ref{stlip}, there is a constant $K>0,$ depending
only on $\varphi$ and $r$ but not on $t,$ such that $| \omega_t| \leq
K$. By taking a subsequence, we may assume that $\omega_t \rightarrow
\omega \in \bar B(0,K)$ as $t\downarrow 0$. Passing to the limit
$t\downarrow 0$ in \EQ{aeneditay} we obtain
\begin{equation*}
0 \leq \omega \cdot D\varphi(x_0) - L(\omega) - H\!\left(D\varphi(x_0)\right).
\end{equation*}
According to the definition of $L,$ we must have equality in the last
inequality, and from this it follows that $\omega \in \partial
H\!\left(D\varphi(x_0)\right)$. Also from the definition of $L$ we
have
\begin{equation*}
0\geq \omega_t \cdot D\varphi(x_0) - L(\omega_t) - H\!\left(D\varphi(x_0)\right),
\end{equation*}
and combining this with \EQ{aeneditay} and dividing by $t,$ we obtain
\begin{equation*}
0\leq  \frac{1}{2} \omega_t\cdot \left( D^2\varphi(x_0)
\omega_t\right) + o( 1 ) \iq{as} t \downarrow 0.
\end{equation*}
By passing to the limit $t\downarrow 0$, we obtain \EQ{aenewts}.

\medskip

We have proved the following result:

\begin{prop}
\label{conv-visco}
If $u\in \lloc(U)$ satisfies the pointwise convexity
criterion \EQ{Pccri}, then $u$ is a viscosity subsolution of the
Aronsson equation \EQ{ae}.
\end{prop}

\subsection*{Acknowledgement}

The third author was partially supported by the Academy of Finland, project \#129784.

\tend